\documentclass[10pt]{article}
\usepackage{mathtools}
\usepackage{graphicx}
\usepackage{amsmath,amsfonts,latexsym,amscd,amssymb,theorem}
\usepackage{graphicx}
\usepackage[english]{babel}
\usepackage[shortlabels]{enumitem}
\usepackage[ansinew]{inputenc}
\usepackage{epsfig}
 \usepackage{epigraph}
\usepackage{amsmath}
\usepackage{amssymb}
\usepackage{romannum}
\usepackage{afterpage}
\usepackage{selinput}
\usepackage{fancyhdr}
\newcommand{\ba}{\begin{eqnarray}}
\newcommand{\ea}{\end{eqnarray}}

\newtheorem{thm}{Theorem}[section]
\newtheorem{problem}{Problem}
\newtheorem{conjecture}{Conjecture}

\newtheorem{theorem}[thm]{Theorem}

\newtheorem{lemma}[thm]{Lemma}
 \newtheorem{assertion}[thm]{Assertion}
  
\newtheorem{proposition}[thm]{Proposition}
\newtheorem{corollary}[thm]{Corollary}

\newtheorem{claim}[thm]{Claim}

\makeatletter
\newcommand*{\rom}[1]{\expandafter\@slowromancap\romannumeral #1@}
\makeatother

\date{}
\begin{document}
\title{About subdivisions of four blocks cycles $C(k_1,1,k_3,1)$ in  \\ digraphs  with large chromatic number}
 	\maketitle
 	\begin{center}\author{Darine Al-Mniny \footnote[1]{KALMA Laboratory, Department of Mathematics, Faculty of Sciences I, Lebanese University, Beirut - Lebanon. (almniny.darine@gmail.com)}$^{,}$\footnote[2]{Department of Mathematics and Physics, School of Arts and Sciences, Lebanese International University, Rayak - Lebanon. (darine.mniny@liu.edu.lb)}, Soukaina Zayat\footnote[1]{KALMA Laboratory, Department of Mathematics, Faculty of Sciences I, Lebanese University, Beirut - Lebanon. (s.zayat@ul.edu.lb)}}$^{,}$\footnote[3]{Department of Computer Science, University of Sciences and Arts in Lebanon, USAL, Beirut - Lebanon. (s.zayat@usal.edu.lb)} \end{center} 
 \begin{abstract}
 	\noindent A cycle with four blocks $C(k_{1}, k_{2},k_{3},k_{4})$ is an oriented cycle formed of four blocks of lengths $k_{1}, k_{2}, k_{3}$ and $k_{4}$ respectively. Recently, Cohen et al. conjectured that for every positive integers  $k_{1}, k_{2}, k_{3}, k_{4}$, there is an integer $g(k_{1},k_{2},k_{3},k_{4})$ such that every strongly connected digraph $D$ containing no subdivisions of $C(k_{1},k_{2},k_{3},k_{4})$  has a chromatic number at most   $g(k_{1},k_{2},k_{3},k_{4})$. This conjecture is confirmed by Cohen et al. for the case of $C(1,1,1,1)$ and by Al-Mniny for the case of $C(k_1,1,1,1)$. In this paper, we affirm Cohen et al.'s  conjecture for the case where $k_2=k_4=1$, namely $g(k_1,1,k_3,1) =O({(k_1+k_3)}^2)$. Moreover, we show that if in addition $D$ is Hamiltonian, then the chromatic number of $D$ is at most $6k$, with $k=\textrm{max}\{k_1,k_3\}.$
 \end{abstract}

\section{Introduction}
\pagenumbering{arabic}
Throughout this paper, all graphs  are considered to be simple, that is, there are no loops and no multiple edges. By giving an orientation to each edge of a graph $G$, the obtained oriented graph is called a digraph. Reciprocally, the graph obtained from a digraph $D$ by ignoring the directions of its arcs is called the underlying graph of $D$, and denoted by $G(D)$ (a circuit of length $2$ in $D$ correspond to one edge in $G(D)$).
The chromatic number of a digraph $D$, denoted by $\chi(D)$, is the chromatic number of its underlying graph. A digraph $D$ is said to be $k$-chromatic if $\chi(D)=k$.   \medbreak

\noindent An oriented path (resp. oriented cycle) is an orientation of a path (resp. cycle). The length of a path (resp. cycle) is the number of its edges.  The order of a path (resp. cycle) is the number of its vertices.  An oriented  path (resp. oriented cycle) is said to be directed if all its arcs have the same orientation.  More formally, an oriented path $P$ whose vertex-set is $V(P)=\{x_1, x_2, ..., x_n\}$ and edge-set is $E(P)=\{(x_i,x_{i+1}); 1 \leq i \leq n-1\}$ is called a directed path. In this case, we write  $P=x_{1}, x_{2}, ..., x_{n}$. Given an oriented path $P$ (resp. oriented cycle $C$), a block is a maximal directed subpath of $P$ (resp. of $C$). We denote by $P(k_{1}, k_{2}, ...,k_{n})$ (resp. $C(k_{1}, k_{2}, ...,k_{n})$) the oriented path (resp. oriented cycle) formed of $n$ blocks of lengths $k_{1}, k_{2}, ...,k_{n-1}$ and $k_{n}$ respectively. \medbreak

\noindent Given a digraph $D$, a directed path (resp. a directed cycle) in $D$ is said to be Hamiltonian if it passes through all the vertices of $D$. If $D$ has a Hamiltonian directed cycle, then $D$ is called a Hamiltonian digraph. Moreover, $D$ is said to be strongly connected if for any two vertices $x$ and $y$ there is a directed path from $x$ to $y$. However, $D$ is said to be acyclic if it contains no directed cycles. Given a digraph $H$, a subdivision of $H$, denoted by $S$-$H$,  is a digraph $H'$  obtained from $H$ by replacing each arc $(x,y)$  by an $xy$-dipath of length at least $1$, all new paths being internally disjoint. If a digraph $D$ does not contain a subdivision of $H$ as a subdigraph, then  $D$ is said to be $H$-subdivision-free.\\

An important question to be asked is the following:
\begin{problem}\label{problem1} 
Which are the graphs $G$ such that every graph with sufficiently high chromatic number contains $G$ as a subgraph?
\end{problem}

\noindent In this context, Erd\"{o}s and Hajnal \cite{edros} proved that every graph with chromatic number at least $k$ contains an odd cycle of length at least $k$. A counterpart of this theorem for even length was settled by Mihok and Schiermeyer \cite{mihok}: Every graph with chromatic
number at least $k$ contains an even cycle of length at least $k$. Further results on graphs with prescribed lengths of cycles have been obtained \cite{gya, nm, lown, mihok, wang}.\\

In their article, Cohen et al. \cite{hjk} investigated a generalization of Problem \ref{problem1} by considering the analogous problem for directed graphs: \begin{problem}\label{problem2} Which are the digraphs $D$ such that every $k$-chromatic digraph contains $D$ as a subdigraph?\end{problem}

\noindent A famous theorem by Erd\"{o}s \cite{lh} states that there exist digraphs with arbitrarily large chromatic number and arbitrarily high girth. This implies that  if $D$ is a digraph containing an oriented cycle, there exist digraphs with arbitrarily high chromatic number  with no subdigraph isomorphic to $D$. Thus the only possible candidates to answer Problem \ref{problem2}  are the oriented trees.  Burr \cite{kg} conjectured that every $(2k-2)$-chromatic digraph contains every oriented tree $T$ of order $k$, and he was able to prove that every $(k-1)^2$-chromatic digraph contains a copy of any oriented tree $T$ of order $k$. The best known  bound, due to Addario-Berry et al. \cite{jhp}, is in $(k/2)^2$. For special oriented trees, better bounds on the chromatic number are known. The most famous one, known as Gallai-Roy theorem, deals with directed paths: 
\begin{theorem}\label{gal}(Gallai \cite{Gallai}, Roy \cite{Roy}) Every $k$-chromatic digraph contains a directed path of length $k-1$.
\end{theorem} 

\noindent However, for paths with two blocks, the best possible upper bound has been determined by Addario-Berry et al. as follows:
\begin{theorem}(Addario-Berry et al. \cite{khj})\label{path} Let $k_{1}$ and $k_{2}$ be positive integers such that $k_{1}+k_{2} \geq 3$. Every  $(k_{1}+k_{2}+1)$-chromatic digraph $D$  contains any  two-blocks path $P(k_{1},k_{2})$.\\
\end{theorem}

The following famous theorem of Bondy  shows that the story does not stop here:
\begin{theorem}(Bondy \cite{fdo}) Every strong digraph $D$ contains a directed cycle of length at least $\chi(D)$.
\end{theorem}

\noindent The strong connectivity assumption is indeed necessary, because there exist  acyclic digraphs (transitive tournaments) with large chromatic number and  no directed cycle. Since any directed cycle of length at least $k$ can be seen as a subdivision of the directed cycle $C_{k}$   of length $k$,   Cohen et al.  conjectured that Bondy's theorem can be extended to all oriented cycles: \begin{conjecture}(Cohen et al. \cite{hjk})\label{conj} For every positive integers $k_1, k_2,...,k_n$, there exists a constant $g(k_1,k_2,...,k_n)$ such that every strongly connected digraph  containing no subdivisions of the oriented cycle  $C(k_1, k_2, ..., k_n)$ has a  chromatic number at most $g(k_1,k_2,...,k_n)$.\end{conjecture}

\noindent Cohen et al. \cite{hjk} noticed that the strongly connected connectivity assumption is also necessary in Conjecture \ref{conj}. This follows from proving the existence of acyclic digraphs with large chromatic number and no subdivisions of $C$ for  any oriented cycle $C$:
\begin{theorem}(Cohen et al. \cite{hjk})
For any positive integers $b,c$, there exists an acyclic digraph $D$ with $\chi(D) \geq c$ in which all oriented cycles have more than $b$ blocks.
\end{theorem}
 
\noindent   In their article,  Cohen et al. \cite{hjk} proved Conjecture \ref{conj}  for the case of  two-blocks cycles. More precisely, they showed that the chromatic number of strong digraphs with no subdivisions of a two-blocks cycle $C(k_{1},k_{2})$ is bounded from above by  $O((k_{1}+k_{2})^4)$:

\begin{theorem} (Cohen et al. \cite{hjk})	Let $k_{1}$ and $k_{2}$ be positive integers such that $k_{1} \geq k_{2} \geq 2$ and $k_{1} \geq 3$. If $D$ is a strong digraph having no subdivisions of $C(k_{1},k_{2})$, then the chromatic number of $D$ is at most  $(k_{1}+k_{2}-2)(k_{1}+k_{2}-3)(2k_{2}+2)(k_{1}+k_{2}+1)$.\end{theorem} 

\noindent   More recently, this bound was  improved by  Kim et al.  as follows: 
\begin{theorem} (Kim et al. \cite{kim}) Let $k_{1}$ and $k_{2}$ be positive integers such that $k_{1} \geq k_{2} \geq 1$ and $k_{1} \geq 2$.  If $D$ is a strong digraph having no subdivisions of $C(k_{1},k_{2})$,  then the  chromatic number of $D$ is at most $2 (2k_1-3)(k_1+2k_2-1)$. \end{theorem}

\noindent In \cite{khj}, Addario et al. asked if  the upper bound of the chromatic number of strongly connected digraphs having no subdivisions of $C(k_1,k_2)$ can be improved to $O(k_1 + k_2)$, which  remains an  open problem.  More recently, Al-Mniny et al. \cite{bispindles and 2-blocks} introduced the notion of secant edges and provided  a positive answer to Addario et al.'s question for the class of digraphs having a Hamiltonian directed path.\\

On the other hand, for the case of four-blocks cycles, Conjecture \ref{conj} is still unresolved unless for some cases. For every positive integers $k_1,k_2,k_3,k_4$, a cycle with four blocks $C(k_{1}, k_{2},k_{3},k_{4})$ is an oriented cycle formed of four blocks of lengths $k_{1}, k_{2}, k_{3}$ and $k_{4}$ respectively. 

 \noindent  In fact,  the restriction of Conjecture \ref{conj} on four-blocks cycles was confirmed by Cohen et al. \cite{hjk}  for the case where $k_1=k_2=k_3=k_4=1$ and by Al-Mniny \cite{4blocks} for the case where $k_1$ is arbitrary and $k_2=k_3=k_4=1$ as follows: 
\begin{theorem}(Cohen et al. \cite{hjk})
	Let $D$ be a strongly connected digraph with no subdivisions of  $C(1,1,1,1)$, then the chromatic number of $D$ is at most $24$.
\end{theorem}

\begin{theorem} (Al-Mniny \cite{4blocks}) Let $k_1$ be a positive integer and let $D$ be a strongly connected digraph  with no subdivisions of $C(k_1,1,1,1)$, then the chromatic number of $D$ is at most $8^{3}.k_1$.
\end{theorem}
More recently, the bound provided by Al-Mniny was improved by Mohsen from $8^{3}.k_1$ to $18.k_1 $: \begin{theorem} (Mohsen \cite{zahraa}) Let $k_1$ be a positive integer and let  $D$ be a strongly connected digraph  with no subdivisions of $C(k_1,1,1,1)$, then the chromatic number of $D$ is at most $18.k_1$.\end{theorem}
 
\noindent In this paper, we confirm  Conjecture \ref{conj} for the four-blocks cycles $C(k_1,1,k_3,1)$, namely $g(k_1,1,k_3,1)$ $=O(k^2)$, with $k=\textrm{max}\{k_1,k_3\}$. Moreover, we provide a linear bound for the chromatic number of Hamiltonian digraphs having no subdivisions of $C(k_1,1,k_3,1)$. \\

 The paper is organized as follows: In Section 2, we introduce some terminologies and notations that will be used throughout the coming sections.  In Section 3, we prove the existence of subdivisions of $C(k_1,1,k_3,1)$ in strong digraphs by using the simple notion of a  maximal-tree and the technique of digraphs decomposing. Then in Section 4,  we reduce the chromatic number obtained in Section 3 for the class of Hamiltonian digraphs having no subdivisions of $C(k_1,1,k_3,1)$. 
 
\section{Preliminaries}
In this section, we introduce some basic definitions and terminologies that will be elementary for the coming sections.\\
 
 In what follows, we denote by $[l]:=\{1,2,...,l\}$ for every positive integer $l$. A graph $G$ is said to be $d$-degenerate, if any  subgraph of $G$ contains a vertex having at most $d$ neighbors. Using an inductive argument, one may easily see the following statement:
\begin{lemma}\label{degenerate}
	If $G$ is $d$-degenerate graph, then $G$ is $(d+1)$-colorable.
\end{lemma} 

 Given two digraphs $D_{1}$ and $D_{2}$, $D_{1} \cup D_{2}$ is defined to be the digraph whose vertex-set is $V(D_{1}) \cup V(D_{2})$ and whose arc-set is $A(D_{1}) \cup A(D_{2})$. The next lemma will be useful for the coming proofs:
 \begin{lemma}\label{far}
 	$\chi(D_{1} \cup D_{2}) \leq \chi(D_{1}) \times \chi(D_{2})$ for any two digraphs  $D_{1}$ and $D_{2}$.
 \end{lemma}
 \begin{proof}
 	For $i \in \{1,2\}$, let $\phi_{i}: V(D_{i}) \longrightarrow \{1, 2, ..., \chi(D_{i})\}$ be a proper $\chi(D_{i})$-coloring of $D_{i}$.
 	Define $\psi$, the coloring of  $V(D_{1} \cup D_{2})$, as follows:
 	\begin{align*}
 	\psi(x) = \left\{ \begin{array}{cc}
 	(\phi_{1}(x), 1) & \hspace{5mm}  x \in V(D_{1}) \setminus V(D_{2});  \\
 	(\phi_{1}(x), \phi_{2}(x)) & \hspace{5mm} x \in V(D_{1}) \cap V(D_{2}); \\
 	(1, \phi_{2}(x)) & \hspace{5mm}  x \in V(D_{2}) \setminus V(D_{1}).\\
 	\end{array} \right.
 	\end{align*}
 	We may easily verify that $\psi$ is a proper coloring of $ D_{1} \cup D_{2}$ with color-set $ \{1, 2, ..., \chi(D_{1})\} \times  \{1, 2, ..., \chi(D_{2})\}$. Consequently, it follows that $\chi(D_{1} \cup D_{2}) \leq \chi(D_{1}) \times \chi(D_{2})$. $\hfill {\square}$ \medbreak
 \end{proof}

\noindent A consequence of the previous lemma is that, if we partition the arc-set of a digraph $D$ into $A_{1}, A_{2},...,A_{l}$, then bounding the chromatic number of all spanning subdigraphs $D_{i}$ of $D$ with arc-set $A_{i}$ gives an upper bound for the chromatic number of $D$.\\

Let $D$ be a digraph. For a dipath or a directed cycle $H$ of $D$  and for any two vertices $u,v$ of $H$, we denote by $H[u,v]$ the subdipath of $H$ with initial vertex $u$ and terminal vertex $v$. Also, we denote by $H[u,v[, H]u,v]$ and $H]u,v[$ the dipaths $H[u,v]-v, H[u,v]-u$ and $H[u,v]-\{u,v\}$, respectively.  Given an oriented cycle $C$ in $D$, a vertex $u$ of $C$ is said to be a source if the two neighbors of $u$ in $C$ are both out-neighbors. If $u$ is a vertex of $D$, we denote by ${N_D}^{+}(u)$ (resp. $N_D^{-}(u)$)  the set of vertices $v$ such that $(u,v)$ (resp. $(v,u)$) is an arc of $D$. The out-degree (resp. in-degree) of $u$, denoted by $d^{+}(u)$ (resp. $d^{-}(u)$), is the cardinality of $N^{+}(u)$ (resp. $N^{-}(u)$). The maximum out-degree of $D$ is defined by $\Delta^{+}(D)= \textrm{max}_{u \in V(D)} d^{+}(u)$. For a vertex $u$ of a graph $G$, we denote by $N_G(u)$ the set of all neighbors of $u$ in $G$, by $d_G(u)$ the cardinality of $N_G(u)$ and by $\delta(G)= \textrm{min}_{u \in V(G)} d_G(u)$.  \\

A tree is a connected graph containing no cycles. An oriented tree is an orientation of a tree. An out-tree is an oriented tree in which all vertices have in-degree at most 1. This implies that an out-tree has exactly one vertex of in-degree $0$, called the source. Given a digraph $D$ having a spanning  out-tree $T$ with source $r$, the level of a vertex $x$  with respect to $T$, denoted by $l_{T}(x)$, is the order of the unique $rx$-directed path in $T$. For a positive integer $i$, we define $L_{i}(T):=\{x \in V(T) | l_{T}(x)=i\}$.  For a vertex $x$ of $D$, the ancestors of $x$ are the vertices that belong to $T[r,x]$. If $y$ is an ancestor of $x$ with respect to $T$, we write $ y \leqslant_{T} x$.  Denoting by $S(x)$ the set of the vertices $y$ of $D$ such that  $x$ is an ancestor of $y$,  $T_{x}$ is defined to be the subtree of $T$ rooted at $x$ and induced by $S(x)$.  For two vertices $x_{1}$ and $x_{2}$ of $D$, the least common ancestor $z$ of $x_{1}$ and $x_{2}$, abbreviated by $\textrm{l.c.a}\{x_1,x_2\}$,   is the common ancestor of $x_{1}$ and $x_{2}$  having the highest level in $T$. Note that the latter notion is well-defined since $r$ is a common ancestor of all vertices. For two vertices $x$ and $y$, we define  $\textrm{min}_T\{x,y\}:=\{x\}$ if  $l_{T}(x) < l_{T}(y)$ and $\textrm{min}_T\{x,y\}:=\{y\}$ if $l_{T}(y) < l_{T}(x)$.    An arc $(x,y)$ of $D$ is said to be forward with respect to $T$ if $l_{T}(x) < l_{T}(y)$. Otherwise,  $(x,y)$ is  called a backward arc.   If for every backward arc $(x,y)$ of $D$  $y \leqslant_{T} x$, then $T$ is called a final out-tree of $D$. In such case,   one may easily see that $D[L_i(T)]$ is an  empty digraph  for all $i \geq 1$. \medbreak 
\noindent 
The next proposition shows an interesting structural property on   digraphs having a spanning out-tree:
\begin{proposition}\label{finaltree}
Given a digraph $D$ having a  spanning out-tree $T$, then $D$ contains a final out-tree.
\end{proposition}
\begin{proof}
Initially, set $T_{0}:=T$. If $T_{0}$ is final, there is nothing to do.  Otherwise, there is an arc $(x,y)$ of $D$ which is backward with respect to $T_{0}$ such that $y$ is not ancestor of $x$.  Let $T_{1}$ be the out-tree obtained from $T_{0}$ by adding  $(x,y)$  to $T_{0}$, and deleting the arc of head $y$ in $T_{0}$. We can easily see that the level of each vertex in $T_{1}$ is at least its level in $T_{0}$, and there exists a vertex ($y$) whose level has strictly increased. Since the level of a vertex cannot increase infinitely, we can see that after a finite number of repeating the above process  we reach an out-tree  which is final.$\hfill {\square}$
\end{proof}

 \section{The existence of $S$-$C(k_1,1,k_3,1)$ in strong digraphs}
  
 From now on, we consider  $k_1$ and $k_3$ to be two positive integers and  $k= \textrm{max} \{k_1, k_3\}$. The aim of this section is to bound from above the chromatic number of strongly connected digraphs having no subdivisions of $C(k,1,k,1)$.  To this end, we consider  $D$ to be a digraph having a final spanning out-tree $T$ rooted at $r$ without subdivisions of $C(k,1, k,1)$. Then we partition the vertex-set of $D$ into subsets $V_1, V_2, ...,V_{2k}$, where   $V_{i}:=\cup_{\alpha \geq 0} L_{i+\alpha (2k)}(T)$ for all  $1 \leq i \leq 2k$.  After that, denoting by $D_{i}$  the subdigraph of $D$ induced by $V_{i}$, we partition the arc-set  of $D_{i}$ as follows:$$A_{1}:=\{(x,y)  | l_{T}(x) < l_{T}(y) \hspace{1mm} \textrm{and} \hspace{1mm} x\leqslant_{T}y\};$$ $$A_{2}:=\{(x,y)  | l_{T}(x) > l_{T}(y) \hspace{1mm} \textrm{and} \hspace{1mm} y \leqslant_{T}x\};$$
 $$A_{3}:=A(D_{i}) \setminus (A_{1} \cup A_{2}).$$
 In the coming sections, we denote by  $D_{i}^{j}$   the spanning subdigraph of $D_{i}$ whose arc-set is $A_{j}$, for $1 \leq i \leq 2k$ and $j=1,2,3$.

 \subsection{Coloring $D_i^1$}

 The main goal of this section is to prove that $\chi(D_i^1)\leq 6$.  To this end, we are going to prove that $D_i^1$ is a $5$-wheel-free digraph. For any integer $k\geq 3$, a $k$-wheel is a graph formed by a cycle $C$ and a vertex $u$ not in $V(C)$, called the center, such that $u$ has at least $k$ neighbors in $C$.  A wheel with a cycle $C$ and a center $u$ is denoted by $(C,u)$. A graph $G$ is said to be $k$-wheel-free graph if it does not contain a $k$-wheel as a subgraph. 
  \begin{theorem}(G.E. Turner \cite{wheels})\label{wheells} For any integer $k\geq 4$, if $G$ is a $k$-wheel-free graph, then $G$ contains a vertex of degree at most $k$.
 \end{theorem}
 Note that the result of Turner in \cite{wheels} is slightly weaker than Theorem \ref{wheells}, but the proof of Turner proves exactly Theorem \ref{wheells} (see \cite{pierre}). Due to an inductive argument, Theorem \ref{wheells} easily implies the following result:
 \begin{corollary}\label{k+1colorable}
 For any integer $k\geq 4$, if $G$ is a $k$-wheel-free graph, then $G$ is $(k+1)$-colorable.\\
 \end{corollary}

 Before going into details, we would like to outline the way we follow  to prove  that  $D_i^1$ is a $5$-wheel-free digraph. The plan is first to reduce the question about the existence of a $5$-wheel with a cycle $C$ in $ D_i^1$ to the existence of a $5$-wheel with a cycle $C$ in a well-defined family $\mathcal{C}$ of cycles (this part will be done in Subsection \ref{landscape} in which we describe the structure of cycles expected to exist in $D_i^1$ according to the number and length of blocks, and according  to the position of the vertices of the cycle with respect to $\leqslant_{T}$). To this end, we prove in Subsection \ref{structural propeties} a very useful lemma that describes the possible positions of the vertices of any three internally disjoint directed paths of $D_i^1$  with respect to $\leqslant_{T}$. Finally, we prove in Subsection \ref{5wheels} that $D_i^1$ is a $5$-wheel-free digraph by considering all the possible positions for the center of the wheel and its neighbors in each expected cycle in $D_i^1$, that is, in each cycle in $\mathcal{C}$.
\subsubsection{Properties of internally disjoint directed paths of $D_i^1$}\label{structural propeties}
In the following, we study  the structural properties of any three internally disjoint directed paths of $D_i^1$. For this purpose, we prove a very useful lemma that our proofs heavily rely on: 
 \begin{lemma}\label{forb}
 Let $R_1=u_{1}, ..., u_{n}$, $R_2=r_{1}, ..., r_{s}$ and $R_3=v_{1}, ..., v_{f}$ be vertex-disjoint directed paths in $D_i^1$ of length at least 1, except possibly $u_{n}=r_s$ or $u_1=r_1$. Then non of the following occurs:
 \newlist{Aenumerate}{enumerate}{1}
\begin{enumerate}[1.]
\item $V(R_1)$ and   $V(R_2)$ are ancestors,  $v_1\leqslant_{T}r_1\leqslant_{T}u_1\leqslant_{T}v_{f}$ ($r_1\neq u_1$),  and one of the below holds:

$a.$ $u_{n}$ $\in T_{v_{f}}$ and $r_s$ $\in T_{v_{f}}$;
 
$b.$ For all $1<j\leq f$ with $r_1 \leqslant_{T} v_j$,  neither $u_{n}$ and $v_j$  are  ancestors nor  $r_s$ and $ v_j$ are  ancestors. 
\item $V(R_1), V(R_2)$ and $V (R_3)$ are ancestors,  $u_{n}\neq r_s$,    $u_{1}$ and $ r_1$ are ancestors of $ v_1$, $ v_1$ is an ancestor of $r_s$ and  $ u_{n}$, and $r_s$ and $u_n$ are ancestors of $v_{f}$.
\item $l(R_j)=1$ for $j=1,2,3$, with $r_1\leqslant_{T}v_1\leqslant_{T}u_1\leqslant_{T}v_2\leqslant_{T}u_2\leqslant_{T}r_2$, $u_{2}\neq r_2$, and $u_1\neq r_1$. 
\item $u_1\leqslant_{T}v_1$, $u_{n}$ and $v_f$ are not ancestors, $\alpha \notin R_1\cup R_3$ with $\alpha=\textrm{l.c.a}\{u_{n},v_f\}$, and  $l(T[\alpha ,u_i])\geq k$  for all $u_i \in T_\alpha.$ 
\end{enumerate} 
 \end{lemma}
 \begin{proof}
 Assume the contrary is true. First, assume that $(1.a)$ holds.  Let $i_1$ and $i_2$ be maximal satisfying $u_{i_1}\leqslant_{T} v_{f}$  and $r_{i_2}\leqslant_{T} v_{f}$. Note that the existence of $u_{i_1}$ and $r_{i_2}$ is guaranteed by the fact that $r_1\leqslant_{T}u_1\leqslant_{T}v_{f}$.  Assume without loss of generality that $r_{i_2} \leqslant_{T} u_{i_1}$. Let $i_3$ be maximal satisfying $v_{i_3}\leqslant_{T}r_{i_2}$  and let $i_4$ be minimal satisfying $u_{i_1}\leqslant_{T}v_{i_4}$.  Possibly, $v_{i_3}=v_1$ and $v_{i_4}=v_f$.   This implies that $T[v_{i_3},r_{i_2}]\cap R_3=\{v_{i_3}\}$ and $T[u_{i_1},v_{i_4}[\cap (R_1\cup R_2\cup R_3)=\{u_{i_1}\}$. If  $r_s = u_n$,  then the union of $T[v_{i_3},r_{i_2}]\cup R_2[r_{i_2},r_s]$, $R_3[v_{i_3},v_{i_4}]$, $T[u_{i_1},v_{i_4}]$ and  $R_1[u_{i_1},u_{n}]$ is a $S$-$C(k,1,k,1)$ in $D$, a contradiction. Else,  assume without loss of generality that $r_s \leqslant_{T} u_n$, and let $i_5$ be chosen to be minimal such that $ r_s \leqslant_{T} u_{i_5}$.  Then the union of $T[v_{i_3},r_{i_2}]\cup R_2[r_{i_2},r_s] \cup T[r_s,u_{i_5}]$, $R_3[v_{i_3},v_{i_4}]$, $T[u_{i_1},v_{i_4}]$ and  $R_1[u_{i_1},u_{i_5}]$ is a $S$-$C(k,1,k,1)$ in $D$, a contradiction. Now assume that $(1.b)$ holds. Since $r_1 \leqslant_{T} v_f$, it follows that $r_s$ and $v_f$ are not ancestors, and $u_n$ and $v_f$ are not ancestors. Consequently,   $\textrm{l.c.a}\{u_{n},v_{f}\} \notin  R_3$.  Let $i_1$ be minimal satisfying  $r_{1}\leqslant_{T}v_{i_1}$. Possibly, $v_{i_1}=v_f$.  Then $v_{i_1}$ and $u_{n}$ are not ancestors,  and  $v_{i_1}$ and $r_s$ are not ancestors.  Let $i_2$ and $i_3$ be maximal satisfying $u_{i_2}\leqslant_{T} v_{i_1}$ and $r_{i_3}\leqslant_{T} v_{i_1}$. Assume without loss of generality that $r_{i_3}\leqslant_{T} u_{i_2}$. This implies that $T[v_{i_1-1},r_{i_3}]\cap R_3=\{v_{i_1-1}\}$ and $T[u_{i_2},v_{i_1}[\cap (R_1\cup R_2\cup R_3)=\{u_{i_2}\}$.  If  $r_s = u_n$, then the union of  $T[v_{i_1-1},r_{i_3}]\cup R_2[r_{i_3},r_s]$, $(v_{i_1-1},v_{i_1})$, $T[u_{i_2},v_{i_1}]$ and $R_1[u_{i_2},u_{n}]$ is a  $S$-$C(k,1,k,1)$ in $D$, a contradiction. Else, assume without loss of generality that  $r_s \leqslant_{T} u_n$, and let  $i_4$ be chosen to be minimal such that $ r_s \leqslant_{T} u_{i_4}$.   Then the union of $T[v_{i_1-1},r_{i_3}]\cup R_2[r_{i_3},r_s] \cup T[r_s,u_{i_4}]$, $(v_{i_1-1},v_{i_1})$, $T[u_{i_2},v_{i_1}]$ and $R_1[u_{i_2},u_{i_4}]$ is a  $S$-$C(k,1,k,1)$ in $D$, a contradiction. Assume now that $(2)$ holds. Let $i_1$ and $i_2$ be minimal satisfying $v_{1}\leqslant_{T} u_{i_1}$ and $v_{1}\leqslant_{T} r_{i_2}$. Assume without loss of generality that $r_{i_2}\leqslant_{T}u_{i_1}$. Let $i_3$ be maximal satisfying $v_{i_3}\leqslant_{T}r_{i_2}$,  and let $i_4$ be minimal satisfying  $u_{i_1}\leqslant_{T}v_{i_4}$. Possibly, $v_{i_3}=v_1$ and $v_{i_4}=v_f$.   This implies that $T]v_{i_3},r_{i_2}[\cap (R_1\cup R_2\cup R_3)=\phi$ and $T]u_{i_1},v_{i_4}[\cap (R_1\cup R_3)=\phi$. If $r_1=u_{1}$,   then the union of $R_1[u_{1},u_{i_1}]\cup T[u_{i_1},v_{i_4}]$, $R_2[r_{1},r_{i_2}]$, $T[v_{i_3},r_{i_2}]$ and $R_3[v_{i_3},v_{i_4}]$ is a $S$-$C(k,1,k,1)$ in $D$, a contradiction. Else, assume w.lo.g that $u_{i_1-1} \leqslant_{T} r_{i_2-1}$. Hence, the union of $R_1[u_{i_1-1},u_{i_1}]\cup T[u_{i_1},v_{i_4}]$, $T[u_{i_1-1}, r_{i_2-1}] \cup R_2[r_{i_2-1},r_{i_2}]$, $T[v_{i_3},r_{i_2}]$ and $R_3[v_{i_3},v_{i_4}]$ is a $S$-$C(k,1,k,1)$ in $D$, a contradiction. Let' assume now that $(3)$ holds, then the union of $T[r_1,v_1]\cup (v_1,v_2)$, $R_2$, $R_1\cup T[u_2,r_2]$ and $T[u_1,v_2]$ is a $S$-$C(k,1,k,1)$ in $D$, a contradiction. Finally if $(4)$ holds, let $i_1$ be maximal satisfying $u_{i_1}\leqslant_{T}v_1$  and let $i_2, i_3$ be minimal satisfying $\alpha \leqslant_{T} v_{i_2}$ and $\alpha \leqslant_{T} u_{i_3}$. Possibly, $u_{i_1}=u_1, v_{i_2}=v_f$ and $u_{i_3}=u_n$. This implies that $T[u_{i_1},v_{1}]\cap R_1=\{u_{i_1}\}$ and $T[\alpha ,v_{i_2}]\cap R_3=\{v_{i_2}\}$. Then the union of $T[u_{i_1},v_1]\cup R_3[v_1,v_{i_2}]$, $R_1[u_{i_1},u_{i_3}]$, $T[\alpha ,u_{i_3}]$ and $T[\alpha ,v_{i_2}]$ is a $S$-$C(k,1,k,1)$ in $D$, a contradiction. This terminates the proof of Lemma \ref{forb}. $\hfill {\square}$ \medbreak 
 \end{proof}

 From now on, we say that $[R_1,R_2,R_3]$ satisfies Lemma \ref{forb}$(1.a)$ (resp. $(1.b), (2), (3)$) if there exist three directed paths $R_1, R_2, R_3$ satisfying the conditions of Lemma \ref{forb}$(1.a)$ (resp. $(1.b), (2), (3)$). Also, we say that $[R_1,R_3]$ satisfies Lemma \ref{forb}$(4)$, if there exist two directed paths $R_1$ and $R_3$ satisfying the conditions of Lemma \ref{forb}$(4)$.  

 \subsubsection{The Landscape of cycles in $D_i^1$}\label{landscape}
This subsection is devoted to reduce the question about the existence of a $5$-wheel with a cycle $C$ in $D_i^1$ to the question about the existence of a $5$-wheel with a cycle $C$ in $ \mathcal{C}$ for a crucial family $\mathcal{C}$  of cycles to be defined below.\\
 
  We first define a special class of cycles $\mathcal{C}$ on at most $8$ blocks in $D_i^1$ by $\mathcal{C} := C_2\cup C_4\cup C_6 \cup C_8$, where $C_2=\lbrace C \in D_i^1; C$ is a $2$-blocks cycle$\rbrace$ and  $C_i$ is the set of cycles in $D_i^1$ with $i$ blocks defined below, for $i=4,6,8$. 
  Now we are going to define the class  $C_i$  of cycles with $i$ blocks for $i=4,6,8$.   To this end, we need to define eight internally disjoint directed paths in $D_i^1$ as follows:  $P_{1}=n_{1},...,n_{t}$; $P_{2}=m_{1},...,m_{l}$; $Q_{1}=x_{1},...,x_{t_{1}}$; $Q_{2}=y_{1},...,y_{l_{1}}$; $Q_{3}=z_{1},...,z_{m}$; $Q_{4}=w_{1},...,w_{r}$; $Q_{5}=c_{1},...,c_{\alpha_{1}}$; $Q_{6}=d_{1},...,d_{\alpha_{2}}$, with $t,l,t_1,l_1,m,r,\alpha_{1},\alpha_{2}\geq 2$.  We advise here  the reader to skip the definitions of $C_i$ exposed below and move  directly to Lemma \ref{structure}. While reading the proof of Lemma \ref{structure}, one can check each cycle and go back to its definition in $\mathcal{C}$.  \medbreak 
 \noindent   Let $C$  be a cycle of $D_i^1$ with at most $8$ blocks.     First, we will define $C_4=\bigcup_{i=1}^{8}C_4^j$, with $C_4^j$ is a class of cycles on $4$ blocks for $j=1,...,8$, and containing cycles with the form $P_1\cup P_2 \cup Q_1 \cup Q_2$. In this case $x_1 = y_1$, $n_t=x_{t_1}$, $n_1=m_1$, and $m_l=y_{l_1}$:
\begin{itemize}[-]
	\item $C_4^1$ $:=\lbrace C; C$ is a $4$-blocks cycle such that $l(Q_1)=l(Q_2)=1$, $n_t$ and $m_l$ are not ancestors with $x_1 \leqslant_{T} n_1=\textrm{l.c.a}\{n_t,m_l\}$, and $l(P_j)\geq 1$ for $j=1,2\rbrace$, 
	\item $C_4^2$ $:=\lbrace C; C$ is a $4$-blocks cycle such that $l(C_4^2)=4$, $n_2$ and $m_2$ are not ancestors with $x_1 \leqslant_{T} n_1 \leqslant_{T} \textrm{l.c.a}\{n_2,m_2\} \}$,
	\item $C_4^3$ $:=\lbrace C; C$ is a $4$-blocks cycle such that $l(Q_1)=l(Q_2)=l(P_2)=1$, $l(P_1)\geq 2$, $n_t$ and $m_2$ are not ancestors with $x_1 \leqslant_{T} n_1 \leqslant_{T} n_{t-1}\leqslant_{T} \textrm{l.c.a}\{n_t,m_2\}\rbrace$,  
	\item $C_4^4$ $:=\lbrace C; C$ is a $4$-blocks cycle such that $l(Q_1)=l(Q_2)=1$, $l(P_j)\geq 2$ for $j=1,2$, $n_t$ and $m_l$ are not ancestors with $x_1 \leqslant_{T} n_1 \leqslant_{T}m_{l-1}\leqslant_{T}n_2\leqslant_{T} n_{t-1}\leqslant_{T} \textrm{l.c.a}\{n_t,m_l\}\rbrace$, 
	\item $C_4^5$ $:=\lbrace C; C$ is a $4$-blocks cycle such that $l(P_2)=l(Q_2)=1$, $l(P_1)\geq 1$, $l(Q_1)\geq 2$, and $x_1 \leqslant_{T} x_{t_1-1} \leqslant_{T}n_1\leqslant_{T}n_t\leqslant_{T} m_2\rbrace$, 
	\item $C_4^6$ $:=\lbrace C; C$ is a $4$-blocks cycle such that $l(Q_1)=l(Q_2)=1$, $l(P_j)\geq 1$ for $j=1,2$, and $x_1 \leqslant_{T} n_1 \leqslant_{T}n_{t}\leqslant_{T}m_2\rbrace$, 
	\item $C_4^7$ $:=\lbrace C; C$ is a $4$-blocks cycle such that $l(P_1)=l(Q_1)=l(Q_2)=1$, $l(P_2)\geq 2$, $x_1 \leqslant_{T} n_1 \leqslant_{T}m_{l-1} \leqslant_{T}n_{2}\leqslant_{T}m_l\rbrace$, 
	\item $C_4^8$ $:=\lbrace C; C$ is a $4$-blocks cycle such that $l(P_1)=l(P_2)=l(Q_2)=1$, $l(Q_1)\geq 2$, and $x_1 \leqslant_{T} n_1 \leqslant_{T}x_2 \leqslant_{T}n_{2}\leqslant_{T}m_2\rbrace$.
\end{itemize} 
Now we will define $C_6=\bigcup_{i=1}^{4}C_6^j$, with $C_6^j$ is a class of cycles on $6$ blocks, for $j=1,...,4$, and containing cycles with the form $P_1\cup P_2 \cup Q_1 \cup Q_2\cup Q_3 \cup Q_4$. In this case  $n_1=m_1$, $y_{l_{1}}=m_l$, $z_{1}=y_1$, $z_{m}=w_r$, $w_1=x_1$, and $x_{t_{1}}=n_t$:
\begin{itemize}[-]
	\item $C_6^1$ $:=\lbrace C; C$ is a $6$-blocks cycle such that $l(P_2)=l(Q_1)=l(Q_2)=l(Q_3)=1$, $l(P_1)\geq 1$, $l(Q_4)\geq 1$, and $y_1 \leqslant_{T} x_1 \leqslant_{T}z_2 \leqslant_{T}n_1 \leqslant_{T} n_{t}\leqslant_{T}m_2\rbrace$, 
	\item $C_6^2$ $:=\lbrace C; C$ is a $6$-blocks cycle such that $l(P_2)=l(Q_2)=l(Q_3)=1$, $l(Q_1)\geq 2$, $l(P_1)\geq 1$, $l(Q_4)\geq 1$, $y_1 \leqslant_{T} x_1 \leqslant_{T}z_2 \leqslant_{T}x_2 \leqslant_{T}x_{t_1-1} \leqslant_{T} n_1\leqslant_{T} n_{t}\leqslant_{T}m_2\rbrace$, 
	\item $C_6^3$ $:=\lbrace C; C$ is a $6$-blocks cycle such that $l(P_2)=l(Q_1)=l(Q_2)=l(Q_3)=l(Q_4)=1$, $l(P_1)\geq 1$, $z_2$ and $m_2$ are not ancestors with $ y_1 \leqslant_{T} x_1 \leqslant_{T}n_1 \leqslant_{T}n_t \leqslant_{T}\textrm{l.c.a}\{z_2,m_2\} \}$, 
	\item $C_6^4$ $:=\lbrace C; C$ is a $6$-blocks cycle such that $l(P_2)=l(Q_2)=l(Q_3)=1$; $l(P_1),l(Q_1),l(Q_4)\geq 1$, $z_2$ and $m_2$ are not ancestors with $x_1=\textrm{l.c.a}\{z_2,m_2\}$, $ y_1 \leqslant_{T} x_1 \leqslant_{T}x_{t_1-1} \leqslant_{T}n_1 \leqslant_{T}n_t \leqslant_{T} m_2$, and $ y_1 \leqslant_{T} x_1 \leqslant_{T}z_2\rbrace$.
\end{itemize}
Now we will define $C_8$, a class of cycles on $8$ blocks, and containing cycles with the form $P_1\cup P_2 \cup(\bigcup_{j=1}^{6} Q_j)$. In this case  $n_1=m_1$, $y_{l_{1}}=m_l$, $z_{1}=y_1$, $z_{m}=w_r$, $w_1=c_1$, $c_{\alpha_{1}}=d_{\alpha_{2}}$, $d_1=x_1$, and $x_{t_{1}}=n_t$.
\begin{itemize}[-]
	\item $C_8:=\lbrace C; C$ is an $8$-blocks cycle such that $C=P_1\cup P_2 \cup(\bigcup_{j=1}^{6} Q_j)$, $l(P_2)=l(Q_2)=l(Q_3)=l(Q_4)=1$, $l(P_1),l(Q_1),l(Q_5),l(Q_6)\geq 1$, $z_2$ and $m_2$ are not ancestors with $x_1=\textrm{l.c.a}\{z_2,m_2\}$, $ y_1 \leqslant_{T} x_1 \leqslant_{T}x_{t_1-1} \leqslant_{T}n_1 \leqslant_{T}n_t \leqslant_{T} m_2$, and $x_1\leqslant_{T} d_{\alpha_{2}-1} \leqslant_{T} w_1 \leqslant_{T} d_{\alpha_{2}} \leqslant_{T}z_2\rbrace$.
\end{itemize}
The following lemma describes the structure of all cycles expected to exist in $D_i^1$, and reduces the question about the existence of a $5$-wheel with a cycle $C$ in $D_i^1$ to the question about the existence of a $5$-wheel with a  cycle $C$ in  $\mathcal{C}$:
 \begin{lemma}\label{structure}
 Let $C$ be a cycle in $D_i^1$, then $C\in \mathcal{C}$.
 \end{lemma}
 \begin{proof}
 If $C$ is a $2$-blocks cycle, then $C\in C_{2}$ and so $C\in \mathcal{C}$. Now assume that $C$ is a cycle with at least $4$ blocks. Let $n_{1}$ be a source of $C$ with maximal level with respect to $\leqslant_{T}$. Let $P_{1}=n_{1},..., n_{t}$, $P_{2}=m_{1}, ..., m_{l}$, $Q_{1}=x_{1}, ..., x_{t_{1}}$, and $Q_{2}=y_{1}, ..., y_{l_{1}}$ be blocks of $C$, with $n_{1}=m_{1}$, $x_{t_{1}}=n_{t}$, $y_{l_{1}}=m_{l}$, and $t,l,t_1,l_1\geq 2$. Clearly, $x_1$ and $y_1$ are sources of $C$. Moreover,  $x_1\leqslant_{T}n_1$ and $y_1\leqslant_{T}n_1$, due to the definition of $D_i^1$ and the maximality of $n_1$.
 \begin{assertion}\label{cla1}
 If $n_{t}$ and $m_{l}$ are not ancestors, then $C\in \bigcup_{j=1}^{4}C_4^j$.
 \end{assertion}
 \noindent \sl {Proof of Assertion \ref{cla1}.} \upshape Let $v=\textrm{l.c.a}\{n_t,m_l\}$.
 \begin{claim}\label{ass1}
	$C$ is a $4$-blocks cycle and $(Q_1 \cup Q_2)\cap T]x_1,n_1[= \phi$.
\end{claim}
\noindent \sl {Subproof. }\upshape Assume by contradiction that this is not the case. Let $s_1$ and $s_2$ be maximal such that $x_{s_1}\leqslant_{T}n_1$ and $y_{s_2}\leqslant_{T}n_1$. Note that in case $C$ is not a $4$-blocks cycle, then $x_1 \neq y_1$ and  possibly $x_{s_1}=x_1$ or $y_{s_2}=y_1$. Otherwise, according to our assumption, we may have either $x_{s_1}=x_1$ or $y_{s_2}=y_1=x_1$ but not both.   Thus,  $x_{s_1}\neq y_{s_2}$.  Assume without loss of generality that $x_{s_1}\leqslant_{T} y_{s_2}$. According to  the choice of $s_1$, it follows that   $x_{s_1+1}\in T]n_1,n_t]$. Consequently,  $[P_2,Q_2[y_{s_2},y_{l_1}],(x_{s_1},x_{s_1+1})]$ satisfies Lemma \ref{forb}$(1.a)$ or Lemma \ref{forb}$(1.b)$, a contradiction. This confirms Claim \ref{ass1}.  $\hfill {\blacklozenge}$ \medbreak 
\noindent 
Therefore, according to Claim \ref{ass1}, we have  $C \cap T[r,n_1[ = \{x_1\}=\{y_1\}$. Moreover, observe that if $v=n_1$ then  $l(Q_1)=1$, since otherwise  the union of $(x_1,x_2)\cup T[x_2,n_j]$, $(x_1,y_2)$, $T[n_1,y_2]$ and $P_1[n_1,n_j]$ is a  $S$-$C(k,1,k,1)$ in $D$, where $j$ is minimal such that $x_2 \leqslant_{T}n_j$. By symmetry,  if $v=n_1$ then   $l(Q_2)=1$ and so $C\in C_4^1$.   Assume now that  $v \neq n_1$. Observe that  $v\notin V(D_i^1)$.  In fact,  if $v\in D_i^1\backslash (Q_1\cup P_2)$, then $[Q_1, P_2]$ satisfies Lemma \ref{forb}$(4)$, a contradiction. Else if $v \in  Q_1\cup P_2$, then $v \in D_i^1\backslash (Q_2\cup P_1)$ and so $[Q_2, P_1]$ satisfies Lemma \ref{forb}$(4)$, a contradiction. Hence, $v\notin V(D_i^1)$.
\begin{claim}\label{ass2}
	$l(Q_1)=l(Q_2)=1$. 
\end{claim}
\noindent \sl {Subproof. }\upshape Assume first  that $Q_1 \cap T[n_1,v] \neq\phi$, and let $y_j$  be the vertex of $Q_2$ satisfying $ n_1 \leqslant_{T}x _2 \leqslant_{T}y_j$.  If $y_j$ and $n_t$ are ancestors, then $[P_1,Q_1[x_2,n_t],Q_2[y_1,y_j]]$ satisfies Lemma \ref{forb}$(1.a)$, a contradiction. This means that  $y_j$ and $n_t$ are not ancestors and so $[P_1,Q_1[x_2,n_t],Q_2[y_1,y_j]]$ satisfies Lemma \ref{forb}$(1.b)$, a contradiction. This proves that $Q_1 \cap T[n_1,v] =\phi$. By symmetry, we have $Q_2 \cap T[n_1,v] = \phi$. Now assume that $Q_1 \cap T]v,n_t[ \neq\phi$. Hence,  $x_2\neq n_t$ and so the union of $(x_1,x_2)\cup T[x_2,n_{j_1}]$, $(y_1,y_2)$, $T[n_{j_2},y_2]$ and  $P_1[n_{j_2},n_{j_1}]$ is a $S$-$C(k,1,k,1)$ in $D$, where $j_1$ is minimal such that $x_2 \leqslant_{T} n_{j_1}$ and $j_2$ is maximal such that $n_{j_2}\leqslant_{T} v$. Thus, $Q_1\cap T]v,n_t[ =\phi$ and by symmetry $Q_2 \cap T]v,m_l[ =\phi$. As a result, $Q_1=(x_1,n_t)$ and $Q_2=(y_1, m_l).$   This yields the desired claim. $\hfill {\blacklozenge}$ \medbreak 
\noindent 
Notice that $l(T[v,n_t])<k$, since else $[(x_1,n_t),P_2]$ satisfies Lemma \ref{forb}$(4)$, a contradiction.  By symmetry, $l(T[v,m_l])<k$. Hence, $P_1\cap T[v,n_t[=\phi$ and $P_2\cap T[v,m_l[=\phi$. If $l(P_1)=l(P_2)=1$, then $C=C_4^2$. Thus let us consider the opposite and assume without loss of generality that $m_{l-1} \leqslant_{T} n_{t-1}$. If $m_{l-1}=n_1$, then $l(P_2)=1$ and so $C\in C_4^3$. Now assume that $m_{l-1}\neq n_1$. This implies that $l(P_1)>1$ and $l(P_2)>1$. Observe that for all $f$ in $P_1\cap T]m_{l-1},n_{t-1}]$, there is no $w$ in $T]n_1,m_{l-1}[$ such that $(w,f)\in A(P_1)$,  since otherwise $[(w,f),Q_1,(m_{j-1},m_j)]$ satisfies Lemma \ref{forb}$(3)$, where $j$ is minimal such that $w \leqslant_{T} m_j$. Hence,  $m_{l-1}\leqslant_{T}n_2$ and so $C\in C_4^4$. This confirms Assertion \ref{cla1}. $\hfill {\lozenge}$  
\begin{assertion}\label{cla2}
	Let $R=\displaystyle{\bigcup_{j=1}^{4}R_j}$ be a $4$-blocks path in $D_i^1$, where $R_1=r_1, ..., r_{s}$, $R_2=u_1, ..., u_{n}$, $R_3=g_1, ..., g_{\kappa}$, and $R_4=v_1, ..., v_h$ are the $4$ blocks of $R$, with $r_{s}=u_{n}$, $g_{1}=u_{1}$, $g_{\kappa}=v_{h}$, $r_1 \neq v_1$, $r_1 \leqslant_{T} u_1$, $v_1 \leqslant_{T} u_1$, and $u_{n} \leqslant_{T} g_{\kappa}$. Then $l(R_3)=1$, $r_{s-1} \leqslant_{T} u_{1}$, and $v_{h-1} \leqslant_{T} r_{1}$.
\end{assertion}
\noindent \sl {Proof of Assertion \ref{cla2}. }\upshape We are going to prove first that $v_{h-1} \leqslant_{T} r_{1}$. Indeed, $R_4\cap T]r_1,u_1[=\phi$, since otherwise $[R_3,R_4[v_j,v_h],R_1]$ satisfies Lemma \ref{forb}$(1.a)$, with $j$ is  minimal such that $r_1 \leqslant_{T} v_j \leqslant_{T} u_1$, a contradiction. Moreover,  $R_4\cap T]u_1,g_{\kappa}[=\phi$, since otherwise $[R_1,R_4[v_1,v_{h-1}],R_3]$ satisfies Lemma \ref{forb}$(2)$, a contradiction. This gives that  $v_{h-1} \leqslant_{T} r_{1}$. Now we want to show that $r_{s-1} \leqslant_{T} u_{1}$. In fact,  $R_1\cap T]u_1,r_{s}[=\phi$. If not,  let $j$ be minimal such that $u_1 \leqslant_{T} r_j \leqslant_{T} r_{s}$ and let $i$ be maximal such that $u_{i} \leqslant_{T} r_j$. According to our assumption together with the previous observation, we get that $[(v_{h-1},v_h),(r_{j-1},r_j),(u_{i},u_{i+1})]$ satisfies Lemma \ref{forb}$(3)$, a contradiction. This proves that  $r_{s-1} \leqslant_{T} u_{1}$. To end the proof, it remains to prove that $l(R_3)=1$. Assume otherwise and consider the possible positions of $g_2$. If $g_2 \leqslant_{T} r_{s}$, let $j$ be maximal satisfying $u_j \leqslant_{T} g_{2}$. Then the union of  $T[u_j,g_2]\cup R_3[g_2,g_{\kappa}]$, $R_2[u_j,u_{n}]$, $T[v_{h-1},r_1]\cup R_1$ and $(v_{h-1},v_h)$ is a $S$-$C(k,1,k,1)$ in $D$, a contradiction. Thus $r_{s-1}\leqslant_{T} g_1 \leqslant_{T} r_s \leqslant_{T} g_2 \leqslant_{T} v_h$ and so $[(v_{h-1},v_h),(r_{s-1},r_{s}),(g_{1},g_2)]$ satisfies Lemma \ref{forb}$(3)$, a contradiction. This implies that $g_2=g_\kappa$ and so $l(R_3)=1$.  $\hfill {\lozenge}$  

\begin{assertion}\label{cl3}
	If $n_{t}$ and $m_{l}$ are ancestors and $C$ is a $4$-blocks cycle, then $C\in \bigcup_{j=5}^{8}C_4^j$.
\end{assertion}
\noindent \sl {Proof of Assertion \ref{cl3}. }\upshape Since $C$ is a $4$-blocks cycle, then $x_1=y_1$. Recall that the maximality of $n_1$ gives that $x_1 \leqslant_{T}n_1$. Assume without loss of generality that $n_t \leqslant_{T} m_{l}$. Note that $Q_2 \cap T]x_1,n_1[ = \emptyset$, since otherwise Assertion \ref{cla2} implies that $y_{l_1-1} \leqslant_{T} x_1$, a contradiction. If $Q_1\cap T]x_1,n_1[ \neq \phi $, then Assertion \ref{cla2} together with the previous remark imply  that $C\in C_4^5$. Let us assume now that the opposite is true. Hence, $n_1 \leqslant_{T} x_2$ and $n_1 \leqslant_{T} y_2$. Clearly,  $l(Q_2)=1$,  since else $l_T(n_1) < l_T(y_2) < l_T(m_l)$ and so $[(y_1,y_2),Q_1,P_2]$ satisfies Lemma \ref{forb}$(2)$, a contradiction. To conclude, we need to prove the following two claims.
\begin{claim}\label{ass3}
	If $P_2\cap T]n_t,m_l[\neq\phi$, then $C\in C_4^6$.
\end{claim}
\noindent \sl {Subproof. }\upshape  Observe first that $P_2 \cap T]n_1,n_t[=\phi$, since else $[Q_2,(n_{i-1},n_{i}),(m_j,m_{j+1})]$ satisfies Lemma \ref{forb}$(3)$, where $j$ is maximal satisfying $m_j \leqslant_{T} n_t$ and $i$ is minimal satisfying $m_j \leqslant_{T} n_{i}$. Moreover, note that  $l(Q_1)=1$, since else the union of $(x_1,x_2)\cup T[x_2,n_{j}], Q_2, (n_1,m_2)\cup T[m_2,m_l]$ and $ P_1[n_1,n_{j}]$ is a $S$-$C(k,1,k,1)$ in $D$, where $j$ is minimal satisfying $x_2 \leqslant_{T} n_{j}$. Hence, $C\in C_4^6$. $\hfill {\blacklozenge}$ \medbreak 
\noindent
\begin{claim}\label{ass4}
	If $P_2\cap T]n_t,m_l[=\phi$, then $C\in \bigcup_{j=6}^{8}C_4^j$.
\end{claim}
\noindent \sl {Subproof. }\upshape We are going to argue on the possible lengths of $P_1$.  If $l(P_1)=1$, then either $l(P_2)>1$ or $l(P_2)=1$. Suppose first that the former holds.  We will prove that $l(Q_1)=1$. Assume else  and consider the possible positions of $x_2$:  If $x_2 \leqslant_{T}m_{l-1}$, then $[Q_2,(m_{i_2-1},m_{i_2}),(x_{i_1},x_{i_1+1})]$ satisfies Lemma \ref{forb}$(3)$, where $i_1$ is maximal satisfying $x_{i_1}\leqslant_{T} m_{l-1}$, and $i_2$ is minimal satisfying $x_{i_1}\leqslant_{T} m_{i_2}$. Else if  $m_{l-1} \leqslant_{T}x_2$, then $[P_1,(m_{l-1},m_l),(x_1,x_2)]$ satisfies Lemma \ref{forb}$(1.a)$, a contradiction. Hence, $l(Q_1)=1$ and so  $C\in C_4^7$. Now assume that the later holds, i.e. $l(P_2)=1$. Then either $l(Q_1)=1$ and so $C\in C_4^6$, or $l(Q_1)>1$ and so $C\in C_4^8$. Else if $l(P_1)>1$, we will prove that   $l(Q_1)=l(P_2)=1$. First assume that $l(Q_1)>1$ and  consider the possible positions of $x_2$: If $ x_2\leqslant_{T}n_{t-1}$, then $[Q_2,(n_{i_2-1},n_{i_2}),(x_{i_1},x_{i_1+1})]$ satisfies Lemma \ref{forb}$(3)$, where $i_1$ is maximal satisfying $x_{i_1}\leqslant_{T} n_{t-1}$, and $i_2$ is minimal satisfying $x_{i_1}\leqslant_{T} n_{i_2}$. Else if  $n_{t-1}\leqslant_{T}x_2$, then $[(n_{t-1},n_t),P_2,(x_1,x_2)]$ satisfies Lemma \ref{forb}$(1.a)$, a contradiction. Hence, $l(Q_1)=1$.   Now assume that $l(P_2)>1$ and consider the possible positions of $m_2$ in $T]n_1,n_t[$: If $m_2\leqslant_{T}n_2$, then  $[(n_1,n_2),Q_1,P_2[m_2,m_l]]$ satisfies Lemma \ref{forb}$(2)$, a contradiction. Else if  $n_2\leqslant_{T}m_2$, then  $[Q_2,(n_1,m_2),(n_{i},n_{i+1})]$ satisfies Lemma \ref{forb}$(3)$, where $i$ is maximal satisfying $n_i\leqslant_{T}m_2$, a contradiction. Hence, $l(P_2)=1$. As a result,  $C\in C_4^6$. This completes the proof of our claim. $\hfill {\blacklozenge}$ \medbreak 
\noindent
In view of what precedes, Assertion \ref{cl3} is confirmed. $\hfill {\lozenge}$ \medbreak 
\noindent
From now on, $V(P_1\cup P_2\cup Q_1 \cup Q_2)$ are considered to be ancestors and $C$ is considered to be  a cycle with at least six blocks. Let $Q_{3}=z_{1},...,z_{m}$ and  $Q_{4}=w_{1},...,w_{r}$ be two other blocks of $C$, with  $z_{1}=y_1$ and $z_{m}=w_r$. Note that  if $C$ is a six-blocks cycle then $w_1=x_1$.  If $C$ is a  cycle with at least ten blocks, then consider $Q_{5}=c_{1},...,c_{\alpha_{1}}$ and $Q_{6}=d_{1},...,d_{\alpha_{2}}$ to be also blocks of $C$, with $w_1=c_1$, $c_{\alpha_{1}}=d_{\alpha_{2}}$ and $d_1=x_1$. In what follows, we will assume without loss of generality that $n_t \leqslant_{T}m_l$. In accordance with Assertion \ref{cla2}, it follows that  $l(P_2) =1$, $y_{l_1-1} \leqslant_{T} x_1$ and $x_{t_1-1} \leqslant_{T} n_1$.\medbreak

\noindent The following observation  will be  very useful for the rest of the proof.
\begin{assertion}\label{cl4}
	Let $(p,q)\in A(D_i^1)$ such that one of the following holds:
	\newlist{Aenumerate}{enumerate}{1}
	\begin{enumerate}[1.]
		\item $p \in T[r,y_{l_1-1}[$ and $q\in T_{y_{l_1-1}}- y_{l_1-1}$.
		\item $p\in T]x_{1},n_1[\backslash Q_1$ and $q\in (T]n_1,m_2[\cup T_{m_2})\backslash (P_1\cup P_2)$.
		\item $p\in T]y_{l_1-1},x_{t_1-1}[$ and $q\in T]x_{t_1-1},m_2[\cup T_{m_2}$.
	\end{enumerate}
	Then $(p,q)\notin A(C)$.
\end{assertion}
\noindent \sl {Proof of Assertion \ref{cl4}. }\upshape Assume else and suppose first that $(1)$ holds. Assume that $ q \leqslant_{T} m_2$. If $q\in T]y_{l_1-1},x_{t_1-1}]$, then the union of $T[p,y_{l_1-1}]\cup (y_{l_1-1},m_2)$, $(p,q)\cup T[q,x_{t_1-1}]\cup (x_{t_1-1}, n_t)$, $T[n_1,n_t]$ and $P_2$ is a $S$-$C(k,1,k,1)$ in $D$, a contradiction. Else if $q\in T]x_{t_1-1},n_t[$, then $[(x_{t_1-1}, n_t),(y_{l_1-1},m_2),(p,q)]$ satisfies Lemma \ref{forb}$(1.a)$, a contradiction. Else if $q\in T]n_t,m_2[$, then $[(x_{t_1-1}, n_t),(p,q),P_2]$ satisfies Lemma \ref{forb}$(2)$, a contradiction. Now assume that  $m_2 \leqslant_{T} q$, then $[P_2,(p,q),(x_{t_1-1}, n_t)]$  satisfies Lemma \ref{forb}$(3)$, a contradiction. This means that $q$ and $m_2$ are not ancestors. Let $\beta =\textrm{l.c.a}\{q,m_2\}$. Observe that $\beta \in T[y_{l_1-1},x_{t_1-1}[$, since otherwise either $\beta \in T[n_1,m_2[$ and so $[P_2,(y_{l_1-1},m_2),(p,q)]$ satisfies Lemma \ref{forb}$(1.b)$, or $ \beta \in T[x_{t_1-1},n_1[$ and so $[(x_{t_1-1}, n_t),(y_{l_1-1},m_2),(p,q)]$ satisfies Lemma \ref{forb}$(1.b)$. Notice that if  $\beta \neq y_{l_1-1}$, then $l(T[\beta,q])<k$, since otherwise $[(p,q),(y_{l_1-1},m_2)]$ satisfies Lemma \ref{forb}$(4)$. Hence, the structure of $C$ and the above discussion imply that there exists $(p_1,q_1)
$ in $A(C)$ such that $\lambda \in T]y_{l_1-1},x_{t_1-1}[$, $p_1 \in T]y_{l_1-1},\lambda[$, $q_1$ and $m_2$ are not ancestors, with $\lambda=\textrm{l.c.a}\{q_1,m_2\}$. Thus, $[(y_{l_1-1},m_2),(p_1,q_1)]$ satisfies Lemma \ref{forb}$(4)$ as $l(T[\lambda,m_2])\geq k$, a contradiction. Assume now that $(2)$ holds. If $q\in T]n_1,m_2[\backslash (P_1\cup P_2)$, then $[(p,q),(x_{t_1-1},n_t),P_2]$ satisfies Lemma \ref{forb}$(2)$, a contradiction. Else if $q\in  T_{m_2}$, then $[P_2,(p,q),Q_1]$  satisfies Lemma \ref{forb}$(1.a)$, a contradiction. To end the proof, assume that $(3)$ holds and consider the possible positions of $q$ in $T]x_{t_1-1},m_2[\cup T_{m_2}$.  If $q\in T]x_{t_1-1},n_t[$, then $[(x_{t_1-1},n_t),(y_{l_1-1},m_2),(p,q)]$ satisfies Lemma \ref{forb}$(3)$. Else if $q\in T]n_t,m_2[$, then $ [(p,q),(x_{t_1-1},n_t),P_2]$ satisfies Lemma \ref{forb}$(2)$. Else  if $m_2 \leqslant_{T} q$, then $ [P_2,(p,q),(x_{t_1-1},n_t)]$ satisfies Lemma \ref{forb}$(3)$, a contradiction. This confirms our assertion. $\hfill {\lozenge}$ \medbreak 
\noindent
 Notice that Assertion \ref{cl4}$(1)$ together with the structure of $C$ imply that $l(Q_2)=1$.
\begin{assertion}\label{cl5}
	If all the vertices of $C$ are ancestors, then  $C\in C_6^1\cup C_6^2$.
\end{assertion}
\noindent \sl {Proof of Assertion \ref{cl5}. }\upshape We will prove a series of claims and conclude.
\begin{claim}\label{ass5}
	$z_2\in T]x_1,x_2[$ and $z_2 \leqslant_{T} n_1$.
\end{claim}
\noindent \sl {Subproof. }\upshape Notice that $z_2\notin T]n_1,m_2[$ since else $[Q_1,(y_1,z_2),P_2]$ satisfies Lemma \ref{forb}$(2)$, a contradiction. Moreover, observe that $z_2\notin T_{m_2}\backslash \lbrace m_2\rbrace$ since else $[P_2,(y_1,z_2),(x_{t_1-1},n_t)]$  satisfies Lemma \ref{forb}$(3)$, a contradiction. This proves that $z_2 \leqslant_{T} n_1$. Now are going to show that $z_2\in T]x_1,x_2[$. Assume first that $l(Q_1)=1$. Then Assertion \ref{cl4}($1$ and $3$) together with the structure of $C$ imply our claim. Assume now that $l(Q_1)>1$. If $z_2 \leqslant_{T}x_1$, then Assertion \ref{cl4}($1$ and $3$) implies that there exists $(p,q)\in A(C)$ such that $p \in T]y_1, x_1[$ and $q\in T]x_1,x_{t_1-1}[$, and so $[(x_j,x_{j+1}),Q_2,(p,q)]$ satisfies Lemma \ref{forb}$(3)$, where $j$ is maximal such that $x_j \leqslant_{T} q$, a contradiction. Else if $z_2 \in T]x_2,n_1[$, then Assertion \ref{cl4} implies that $z_2 \notin T]x_{t_1-1},n_1[$ and so  $z_2 \in T]x_2,x_{t_1-1}[$.  If there exists $(p,q)\in A(C)$ such that $p \in T]y_1, x_1[$ and $q\in T]x_1,x_{t_1-1}[$, then $[(x_j,x_{j+1}),Q_2,(p,q)]$ satisfies Lemma \ref{forb}$(3)$, where $j$ is maximal such that $x_j \leqslant_{T} q$, a contradiction. Combining what precedes together with Assertion \ref{cl4}($1$ and $3$) and the structure of $C$,  we guarantee the existence of an arc  $(p,q)$ of $C$ such that $p\in T[x_1,x_2[$ and $q\in T]x_2,x_{t_1-1}[$. Hence, $[(x_{j},x_{j+1}),Q_2,(p,q)]$ satisfies Lemma \ref{forb}$(3)$, where $j$ is maximal such that $x_{j} \leqslant_{T} q$, a contradiction. This proves that 	$z_2\in T]x_1,x_2[$ and thus confirms our claim. $\hfill {\blacklozenge}$ \medbreak 
\noindent
\begin{claim}\label{ass6}
	For all $p\in T[x_1,z_2[$, there exists no vertex $q\in T]z_2,n_1[$ such that $(p,q)\in A(C)$.
\end{claim}
\noindent \sl {Subproof. }\upshape Assume otherwise. Then $[Q_1,(p,q),(y_1,z_2)]$ satisfies Lemma \ref{forb}$(1.a)$, a contradiction. $\hfill {\blacklozenge}$ \medbreak 
\noindent
In view of Assertion \ref{cl4}, Claim \ref{ass6} and Lemma \ref{forb}$(3)$, one may easily see that $l(Q_3)=1$. Consequently,  Assertion \ref{cl4} and Lemma \ref{forb}$(3)$ imply  that $w_1 \in T[x_1,z_2[$. In what follows, assume  that $C$ is not a $6$-blocks cycle, that is, $w_1\neq x_1$.
\begin{claim}\label{ass7}
	For all $p\in T[x_1,w_1[$, there exists no vertex $q\in T]w_1,z_2[$ such that $(p,q)\in A(C)$.
\end{claim}
\noindent \sl {Subproof. }\upshape Assume else and let $j$ be maximal such that $w_j \leqslant_{T} q$. Then $[(w_j,w_{j+1}),Q_2,(p,q)]$ satisfies Lemma \ref{forb}$(3)$, a contradiction. $\hfill {\blacklozenge}$ \medbreak 
\noindent
In view of Assertion \ref{cl4}, Claims \ref{ass6} and \ref{ass7} and Lemma \ref{forb}$(3)$, one may easily see that the structure of $C$ induces the  existence of the arc $(x_1,q)$ in $A(C)$ for some  $q\in T]w_1,z_2[$,  which contradicts Claim \ref{ass7}. As a result, $C$ is a $6$-blocks cycle and so $x_1=w_1$. Hence, $C\in C_6^1\cup C_6^2$. This completes the proof of Assertion \ref{cl5}. $\hfill {\lozenge}$ \medbreak

\noindent In what follows, we denote by $C=h_1, h_2, ...., h_\delta, h_1$. 
\begin{assertion}\label{cl6}
	If there exist vertices $h_{j_1}$ and $h_{j_2}$ of $C$ for some $j_1,j_2 \in \lbrace 1,...,\delta\rbrace $ such that $h_{j_1}$ and $h_{j_2}$ are not ancestors, then $C\in C_6^3\cup C_6^4\cup C_8$.
\end{assertion}
\noindent \sl {Proof of Assertion \ref{cl6}. }\upshape We will prove a series of claims.
\begin{claim}\label{ass8}
	Let $q_1 \in V(D_i^1)$ such that $m_2$ and $q_1$ are not ancestors and $v^*=\textrm{l.c.a}\{m_2,q_1\} \in T_{y_1}-y_{1}$. Let $p_j \in T]y_1,v^*]$ for $j=1,2$ with $p_1\neq p_2$, and let $q_2 \in \displaystyle{\bigcup_{z\in T]v^*,q_1]}T_z}$. If $(p_{1},q_{1})\in A(C)$, then $(p_{2},q_{2})\notin A(C)$.
\end{claim}
\noindent \sl {Subproof. }\upshape Suppose otherwise and assume without loss of generality that $p_1\leqslant_{T}p_2$. Notice that $ v^*\in T]n_t,m_2[$, since else $[Q_2,(p_1,q_1)]$ satisfies Lemma \ref{forb}$(4)$ as $l(T[v^*,m_2])\geq k$. If $q_1$ and $q_2$ are ancestors (possibly $q_1=q_2$), then $[(p_1,q_1),(p_2,q_2),Q_2]$ satisfies Lemma \ref{forb}$(1.b)$, a contradiction. This means that if such arcs exist in $C$ then $q_1$ and $q_2$ are not ancestors and so the structure of $C$ implies that $l(T[v^*,q_j])\geq k$ for $j=1,2$. Now we are going to show that $p_j \in T]n_1,n_t[$ for $j=1,2$.  Notice first that $n_1 \leqslant_{T}p_1$, since otherwise $[(p_1,q_1),P_2]$ satisfies Lemma \ref{forb}$(4)$, a contradiction. Now note  that $p_2\neq v^*$, since else $[P_2,(p_1,q_1)]$ satisfies Lemma \ref{forb}$(4)$, a contradiction. Moreover, observe that $p_2 \leqslant_{T} n_t$, since otherwise the union of $T[v^*,q_2]$, $T[v^*,m_2]$, $T[x_{t_1-1},n_1]\cup P_2$ and  $(x_{t_1-1},n_t)\cup T[n_t,p_2]\cup (p_2,q_2)$ is a $S$-$C(k,1,k,1)$ in $D$, a contradiction. Thus, $p_j \in T]n_1,n_t[$ for $j=1,2$. Consequently, the structure of $C$ induces the existence of an arc $(f_1,f_2)\in A(C)$, such that either $f_1 \in T]n_1,n_t[$ and $f_2 \in T]n_t,m_2[\cup T_{m_2}\backslash \lbrace m_2\rbrace$, or $f_1 \leqslant_{T}n_1$ and $f_2 \in T]n_1,n_t[$. If the former holds, then  $[P_2,(f_1,f_2),(x_{t_1-1},n_t)]$ satisfies Lemma \ref{forb}$(1.a)$, a contradiction. Else if the latter holds, then $[(f_1,f_2),(x_{t_1-1},n_t),P_2]$ satisfies Lemma \ref{forb}$(2)$, a contradiction. This completes the proof of our claim. $\hfill {\blacklozenge}$ \medbreak 
\noindent
Let $\beta$ be minimal such that $h_\beta$ and $m_2$ are not ancestors. According to Claim \ref{ass8},  it follows that  $h_\beta =z_2$. Let  $v^*=\textrm{l.c.a}\{m_2,z_2\}$. Indeed, Assertion \ref{cl4}$(1)$ together with the structure of $C$ imply that $v^*\neq y_1$ and so induce  the existence of  an arc $(p,q)\in A(C)$ such that $p\in T]y_1,v^*]$ and $q\in \displaystyle{\bigcup_{z\in T]v^*,z_2]}T_z}$. 
\begin{claim}\label{ass9}
	If $v^*\in T]n_t,m_2[$, then $C\in C_6^3$.
\end{claim}
\noindent \sl {Subproof. }\upshape Notice first that $l(T[v^*,z_2])<k$, since else $[(y_1,z_2),P_2]$   satisfies Lemma \ref{forb}$(4)$, a contradiction. This implies that $q\in T_{z_2}$ and $p \neq v^*$.   In fact,  $q=z_2$, since else the union of $(y_1,z_2)\cup T[z_2,q]$, $Q_2$, $T[p,m_2]$ and $(p,q)$ is a  $S$-$C(k,1,k,1)$ in $D$, a contradiction. This gives that $p=h_{\beta +1}=w_{r-1}$.  Now we will study the position of $p$. If $p\in T]x_{t_1-1},n_1[$, then the union of $T[p,n_1]\cup P_2$, $(p,z_2)$, $T[y_1,x_{t_1-1}]\cup (x_{t_1-1},n_t)\cup T[n_t,z_2]$ and $Q_2$ is a $S$-$C(k,1,k,1)$ in $D$, a contradiction. Else if $p\in T]n_1,v^*[$, then the maximality of $n_1$ implies that $w_1 \leqslant_{T}n_1$, and so $[(x_{t_1-1},n_t),Q_4[w_1,w_{r-1}],P_2]$ satisfies Lemma \ref{forb}$(2)$, a contradiction. Else if $p\in T]y_1,x_1[$,  then Assertion \ref{cl4}$(1)$, Assertion \ref{cl4}$(3)$ and Claim \ref{ass8} imply that there exists an arc $(h,h')\in A(C)$ such that $h\in T]y_1,x_1[$ and $h'\in T]x_1,x_{t_1-1}[$, and so $[(x_j,x_{j+1}),Q_2,(h,h')]$ satisfies Lemma \ref{forb}$(3)$, where $j$ is maximal satisfying $x_j \leqslant_{T} h'$, a contradiction. Else if $p\in T[x_1,x_{t_1-1}[$ and $t_1-1\neq 1$, then the union of $T[y_1,p]\cup (p,z_2)$, $Q_2$, $T[x_{t_1-1},n_1]\cup P_2$ and $(x_{t_1-1},n_t)\cup T[n_t,z_2]$ is a $S$-$C(k,1,k,1)$ in $D$, a contradiction. Then $p=x_{t_1-1}=x_1$, and so $C\in C_6^3$. This completes the proof. $\hfill {\blacklozenge}$ 
\begin{claim}\label{ass10}
	If $v^*\notin T]n_t,m_2[$, then $v^*=p=x_1$.
\end{claim}
\noindent \sl {Subproof. }\upshape Since $v^*\notin T]n_t,m_2[$, then $v^*\in T]y_1,n_t]$ and so clearly $l(T[v^*,m_2])\geq k$. Observe that  $p= v^*$, since else  $[Q_2,(p,q)]$ satisfies Lemma \ref{forb}$(4)$, a contradiction. This  gives that  $l(T[v^*,q])\geq k$. If $v^*\in T]x_1,n_t[$, then $[(y_1,z_2),Q_1]$ satisfies Lemma \ref{forb}$(4)$, a contradiction.  Else if $v^*\in T]y_1,x_1[$, then Assertion \ref{cl4}$(1)$, Assertion \ref{cl4}$(3)$ and Claim \ref{ass8} imply that there exists an arc $(h,h')\in A(C)$ such that $h\in T]y_1,x_1[$ and $h'\in T]x_1,x_{t_1-1}[$, and so $[(x_j,x_{j+1}),Q_2,(h,h')]$ satisfies Lemma \ref{forb}$(3)$, where $j$ is maximal satisfying $x_j \leqslant_{T} h'$, a contradiction. Hence, $v^*=p=x_1$. This confirms our claim. $\hfill {\blacklozenge}$ \medbreak 
\noindent
From now on, we will assume that $v^*=p=x_1$, since else $C\in C_6^3$, due to Claim \ref{ass10} and Claim \ref{ass9}.
\begin{claim}\label{ass11}
	For all $z\in T_{z_2}- z_2$, there exists no vertex $w\in T[x_1,z_2[$ such that $(w,z)\in A(C)$.
\end{claim}
\noindent \sl {Subproof. } \upshape Assume the contrary is true. Then the union of $(y_1,z_2)\cup T[z_2,z]$, $Q_2$, $T[x_1,m_2]$ and $T[x_1,w]\cup (w,z)$ is a $S$-$C(k,1,k,1)$ in $D$, a contradiction. $\hfill {\blacklozenge}$ \medbreak 
\noindent
Now Claims \ref{ass11} and \ref{ass8}, Assertion \ref{cl4} and  Lemma \ref{forb}$(3)$ imply that $l(Q_3)=1$ and for all $j\geq \beta +1$, $h_j\in  \{x_1\} \cup  \displaystyle{\bigcup_{z\in T]x_1,z_2[}T_z}$.  Now it remains to prove two claims and conclude.
\begin{claim}\label{ass12}
	If  $h_j$ and $z_2$ are ancestors for all $j\geq \beta +1$, then $C\in C_6^4\cup C_8$. 
\end{claim}
\noindent \sl {Subproof. }\upshape Assume first that $C$ is a $6$-blocks cycle.  Then $w_1=x_1$ and so $C\in C_6^4$ as $l(Q_3)=1$. Now assume that $C$ is a cycle with at least eight blocks. If $C$ is an $8$-blocks cycle, then Claim \ref{ass11} implies that $c_{\alpha_{1}}\leqslant_{T}z_2$ and $d_1=x_1$. As $d_1 \leqslant_{T}w_1$, then Assertion \ref{cla2} implies that $l(Q_4)=1$ and $d_{\alpha_{2}-1} \leqslant_{T}w_1$. Hence, $C\in C_8$. Let us assume now that $C$ is a cycle with at least ten blocks. Then $d_1\neq x_1$. If $d_1 \leqslant_{T}w_1$, then Assertion \ref{cla2} implies that $l(Q_4)=1$ and $d_{\alpha_{2}-1} \leqslant_{T}w_1$, $c_{\alpha_{1}}\leqslant_{T}z_2$ and $v^*\leqslant_{T}d_1$. Observe that for all $f\in T[x_1,d_{\alpha_{2}-1}[$, there exists no $w\in T]d_{\alpha_{2}-1},z_2[$ such that $(f,w)$ is an arc of $C$.  Assume else, then either $w\in T]d_{\alpha_{2}-1},c_{\alpha_{1}}[$ and so $[(d_{\alpha_{2}-1},c_{\alpha_{1}}),Q_3,(f,w)]$ satisfies Lemma \ref{forb}$(3)$, or $w\in T]c_{\alpha_{1}},z_2[$ and so $[(f,w),Q_6,Q_4]$ satisfies Lemma \ref{forb}$(2)$, a contradiction. Moreover, observe that for all $f\in T[x_1,d_{1}[$, there exists no $w\in T]d_1,d_{\alpha_{2}-1}[$ such that $(f,w)$ is an arc of $C$, since else $[(d_{j_1},d_{j_1+1}),(y_1,z_2),(f,w)]$  satisfies Lemma \ref{forb}$(2)$, where $j_1$ is maximal satisfying $d_{j_1}\leqslant_{T}w$, a contradiction.   In view of  these two observations, it follows  that $h_{\delta}\in T]d_1,z_2[$ and so $(x_1,h_{\delta})\notin A(C)$, a contradiction. Thus, $w_1 \leqslant_{T}d_1$. Again we will notice two observations. For all $f\in T]c_1,c_{\alpha_{1}}[$, there exists no vertex $w\in T]c_{\alpha_{1}},z_2[$ such that $(f,w)$ is an arc of $C$, since else $[(f,w),Q_3,(c_{j_1-1},c_{j_1})]$  satisfies Lemma \ref{forb}$(3)$, where $j_1$ is minimal satisfying $f\leqslant_{T}c_{j_1}$. Also observe that for all $f\in T[x_1,c_1[$, there exists no vertex $w\in T]c_1,c_{\alpha_{1}}[$ such that $(f,w)$ is an arc of $C$,  since else $[(c_{j_2},c_{j_2+1}),Q_3,(f,w)]$  satisfies Lemma \ref{forb}$(3)$, where $j_2$ is maximal satisfying $c_{j_2}\leqslant_{T}w$. Hence, $h_\delta \in T]w_1,c_{\alpha_{1}}[$ and so $(x_1,h_{\delta})\notin A(C)$, a contradiction. This completes the proof. $\hfill {\blacklozenge}$

\begin{claim}\label{ass13}
	For all $j\geq \beta +1$, $h_j$ and $z_2$ are ancestors.
\end{claim}
\noindent \sl {Subproof. }\upshape Assume the contrary is true. Let $i>\beta +1$ be minimal such that $h_i$ and $z_2$ are not ancestors. Then $h_{i-1}\in T]x_1,z_2[\cap C$ and $(h_{i-1},h_i)\in A(C)$. Set $x=\textrm{l.c.a}\{h_i,z_2\}$. The structure of $C$ implies that there exists an arc $(h,h^*)$ of $C$ such that $h\in T[x_1,x]\backslash\lbrace h_{i-1}\rbrace$ and $h^*\in \displaystyle{\bigcup_{z\in T]x,h_i]}T_z}$. Assume that $(h,h^*)$ is chosen to be the first arc of $C$ with this property. Clearly,  $x\notin V(D_i^1)$, since otherwise $[Q_3,(h,h^*)]$ or $[Q_3,(h_{i-1},h_i)]$ satisfies Lemma \ref{forb}$(4)$, a contradiction.  Observe that 	for all $z\in T[x_1,x[\backslash\lbrace h_{i-1}\rbrace$, there exists no vertex $w\in T]x,h_i]\cup T_{h_i}$ such that $(z,w)\in A(C)$, since otherwise $[(z,w),(h_{i-1},h_i),Q_3]$ satisfies Lemma \ref{forb}$(1.b)$, a contradiction. This implies that $h^*$ and $h_i$ are not ancestors. Let $\rho >i$ be minimal such that $h_\rho$ and $h_i$ are not ancestors and let $\gamma =\textrm{l.c.a}\{h_\rho ,h_i\}$. Clearly, $h_{\rho -1}\in T]x,\gamma ]$ as $x\notin V(D_i^1)$. Moreover, the definition of $D_i^1$ and the structure of $C$ imply that there exists $i_1$, with $i\leq i_1<\rho -1$, $h_{i_1}\in T]\gamma ,h_i]\cup T_{h_i}$ and $h_{i_1+1}\in T]x,\gamma]$ such that $(h_{i_1+1},h_{i_1})\in A(C)$ (possibly $h_{i_1}=h_i$ and $h_{i_1+1}=h_{\rho -1}$). Indeed, 	for all $z\in T]\gamma ,h_{\rho}]\cup T_{h_{\rho}}$, there exists no vertex $w\in T]h_{i-1},\gamma ]\backslash \{h_{\rho -1}\}$ such that $(w,z)\in A(C)$, since  else $[(w,z),(h_{\rho -1},h_{\rho}),(h_{i-1},h_i)]$ satisfies Lemma \ref{forb}$(1.b)$, a contradiction. 
Furthermore, 	for all $z\in V(D_i^1)$ such that $z$ and $h_{i_1}$ are not ancestors and $\textrm{l.c.a} \{z,h_{i_1}\}\in T[h_{i_1+1},h_{i}[$, there exists no vertex $w\in T[x_1,h_{i-1}[$ such that $(w,z)\in A(C)$, since else $[(h_{i-1},h_i),(h_{i_1+1},h_{i_1}),(w,z)]$ satisfies Lemma \ref{forb}$(1.b)$.
In view of these observations together with the structure of $C$, we guarantee the existence of an arc $(w,z)\in A(C)$ such that $z$ and $h_{\rho}$ are not ancestors, $w_1=\textrm{l.c.a}\{z,h_{\rho}\}\in T]\gamma ,h_{\rho}[$ and $w\in T]h_{i-1},\gamma ]$. Let $j$ be minimal such that $(h_{j+1},h_j)$ satisfies the properties of $(w,z)$.
Notice that $l(T[\gamma ,h_i])<k$, since otherwise $[(h_{i-1},h_i),(h_{\rho -1},h_{\rho})]$ or $[(h_{i-1},h_i),(h_{j+1},h_j)]$ satisfies Lemma \ref{forb}$(4)$, a contradiction. Thus $\gamma \notin V(D_i^1)$, $h_{i_1}\in T_{h_i}$ and $h_{i_1+1}\in T]x,\gamma [$. One may easily check that  the position of $h_j$, the structure of $C$ and the definition of $D_i^1$ imply that $l(T[\gamma ,h_j])\geq k$ and $l(T[\gamma ,h_\rho])\geq k$. This gives that $h_{j+1}\in T]h_{i_1+1},\gamma[$ and $h_{\rho}\in T[h_{i_1+1},\gamma[$, since otherwise $[(h_{j+1},h_j), (h_{i_1+1}, h_{i_1})]$ or $[(h_{\rho -1},h_\rho), (h_{i_1+1}, h_{i_1})]$ satisfies Lemma \ref{forb}$(4)$, a contradiction. Now all the above explanation implies that if $\gamma_{1}=\textrm{l.c.a}\{h^* ,h_i\}\in T[h_{i_1+1},h_i[$, then $h\in T]h_{i_1+1},\gamma_{1}[$, a contradiction. Hence, $\gamma_{1}\in T]x,h_{i_1+1}[$. Notice that if $h_{i-1}\leqslant_{T}h$, then $[(h_{i-1},h_i),(h,h^*)]$ satisfies Lemma \ref{forb}$(4)$, a contradiction. Hence, $h\leqslant_{T}h_{i-1}$ and so $l(T[\gamma_{1},h^*])<k$, since otherwise $[(h,h^*),(h_{i-1},h_i)]$ satisfies Lemma \ref{forb}$(4)$, a contradiction. Now the minimality of $(h,h^*)$, the fact that $\gamma_{1}\notin V(D_i^1)$, the structure of $C$ and the definition of $D_i^1$ imply that  there exists an arc $(p_1,q_1)$ of $C$ such that  $p_1\in T]x,\gamma_{1}[$ annd $q_1\in \displaystyle{\bigcup_{z\in T]\gamma_{1},h^*]}}T_z$. Then $[(h_{i-1},h_i),(p_1,q_1)]$ satisfies Lemma \ref{forb}$(4)$, a contradiction.  This completes the proof of Claim \ref{ass13}.  $\hfill {\blacklozenge}$ \medbreak 
\noindent
Therefore, Claims \ref{ass8}, \ref{ass9}, \ref{ass10}, \ref{ass11}, \ref{ass12}, and \ref{ass13} complete the proof of Assertion \ref{cl6}.    $\hfill {\lozenge}$ \medbreak 
\noindent
In the light of  all the above Assertions,   Lemma \ref{structure} is proved.  $\hfill {\square}$
\end{proof}
\subsubsection{The existence of $5$-wheels in $D_i^1$}\label{5wheels}
In this subsection, we provide an upper bound for the chromatic number of $D_i^1$ and complete the proof  by proving that $D_i^1$ is a $5$-wheel-free digraph. 
\begin{proposition}\label{prop1}
$\chi(D_i^1) \leq 6$  for all $i \in \{ 1, ...,2 k\}$.
\end{proposition}
\begin{proof}
	Assume to the contrary that  $\chi(D_i^1) > 6$. Then Corollary \ref{k+1colorable} implies that  $D_i^1$ contains a $5$-wheel with cycle $C$ and center $\omega$, denoted by $W= (C,\omega )$. Let $\{a_1,...,a_5\}\subseteq N_C(\omega )$. By Lemma \ref{structure}, $C\in \mathcal{C}$.  Clearly $C\notin C_4^2$, since $W$ is a $5$-wheel and the cycles in $C_4^2$ are of $4$ vertices. We will prove series of claims and conclude.
	\begin{claim}\label{cl7}
		$C\notin C_2$.
	\end{claim}
	\noindent \sl {Subproof.} \upshape Assume to the contrary that $C\in C_2$, and assume without loss of generality that $a_1 \leqslant_{T}a_2\leqslant_{T}a_3\leqslant_{T}a_4\leqslant_{T}a_5$. Let $P_{1}=n_1,...,n_{t}$, $t\geq 2$; $P_{2}=m_{1},...,m_{l}$, $l\geq 2$ with $n_1=m_1$ and $n_t=m_l$ be the blocks of $C$. Notice that if $\omega \leqslant_{T}a_1 $, or $a_5\leqslant_{T}\omega$, then there exist at least two vertices in $\{a_2,a_3,a_4\}$ that belongs to the same block of $C$. Assume without loss of generality that $a_2$ and $a_4$ are vertices of $P_1$. Let $i_1$ be maximal satisfying $m_{i_1} \leqslant_{T}a_2$, and let $i_2$ be minimal satisfying $a_4\leqslant_{T}m_{i_2}$. Assume first that $\omega \leqslant_{T}a_1$. Then either $\omega \leqslant_{T}m_{i_1}$ and so $[(\omega ,a_2),(\omega ,a_4),P_2[m_{i_1},m_l]]$  satisfies Lemma \ref{forb}$(2)$, or $m_{i_1}\leqslant_{T}\omega$ and so $P_2\cap T[\omega ,a_2]=\phi$ and the union of $(\omega ,a_4)\cup T[a_4,m_{i_2}]$, $(\omega ,a_2)$, $P_1[m_1,a_1]\cup T[a_1,a_2]$ and $P_2[m_1,m_{i_2}]$ is a $S$-$C(k,1,k,1)$ in $D$, a contradiction. Now assume that $a_5\leqslant_{T}\omega$. If $m_{i_2} \leqslant_{T}\omega $, or $m_{i_2}$ and $\omega$ are not ancestors, then the union of $T[m_{i_1},a_2]\cup (a_2,\omega )$, $P_2[m_{i_1},m_{i_2}]$, $T[a_4,m_{i_2}]$ and $(a_4,\omega )$ is a $S$-$C(k,1,k,1)$ in $D$, a contradiction. And if $\omega\leqslant_{T}m_{i_2}$, then $P_2\cap T[a_4,\omega ]=\phi$ and so the union of $T[m_{i_1},a_2]\cup (a_2,\omega )$, $P_2[m_{i_1},m_{l}]$, $T[a_4,a_5]\cup P_1[a_5,n_t]$ and $(a_4,\omega )$ is a $S$-$C(k,1,k,1)$ in $D$, a contradiction. So either $3\leq $ $\mid$$N^+_C(\omega )$$\mid$ $\leq 4$, or $3\leq $ $\mid$$N^-_C(\omega )$$\mid$ $\leq 4$. Assume that the former holds and let $a_{j_1}\leqslant_{T}a_{j_2}\leqslant_{T}a_{j_3}$ be distinct out-neighbors of $\omega$ in $C$, and let $a_{j_4}$ be an in-neighbors of $\omega$ in $C$. Assume without loss of generality that $a_{j_1}\in P_1$. We are going to prove that $m_2\in T_{a_5}$, $\omega \leqslant_{T}n_2$, and so $\mid$$N^+_C(\omega )$$\mid$ $= 4$. Notice that $P_2\cap T]\omega ,a_5[=\phi$, since else $[P_1[n_1,a_{j_1}],P_2[n_1,m_{i}],(\omega ,a_5)]$ satisfies Lemma \ref{forb}$(2)$, where $i$ is minimal satisfying $\omega \leqslant_{T}m_{i}$. Observe that $P_2\cap T]n_1,\omega [=\phi$, since else $[P_2[m_2,m_l],(\omega ,a_{j_3}),P_1[n_1,a_{j_1}]]$ satisfies Lemma \ref{forb}$(1.a)$. So $m_2\in T_{a_5}$. Clearly, $\omega \leqslant_{T}n_2$ since else, $[(\omega ,a_{j_2}),(m_1,m_2),(n_{i},n_{i+1})]$  satisfies Lemma \ref{forb}$(3)$, where $i$ is maximal satisfying $n_{i} \leqslant_{T}\omega$. Then $a_{j_4}=n_1=m_1$. Since $m_2\in T_{a_5}$, assume without loss of generality that $a_4=a_{j_3}\leqslant_{T}a_{j_5}=a_5$. Then the union of $(\omega ,a_{j_3})\cup T[a_{j_3},m_2]$, $(\omega ,a_{j_2})$, $(m_1,n_2)\cup T[n_2,a_{j_2}]$ and $(m_1,m_2)$ is a $S$-$C(k,1,k,1)$ in $D$, a contradiction. So the latter holds. Let $a_{j_1}\leqslant_{T}a_{j_2}\leqslant_{T}a_{j_3}$ be distinct in-neighbors of $\omega$ in $C$, and let $a_{j_4}$ be an out-neighbor of $\omega$ in $C$. Assume without loss of generality that $a_{j_3}\in P_1$, and let $i$ is maximal satisfying $m_{i} \leqslant_{T}\omega$. We are going to prove that $m_{l-1}\in T[r,a_1]$ and $\mid$$N^-_C(\omega )$$\mid$ $= 4$. Notice that $P_2\cap T]a_1,\omega [=\phi$, since else $[P_1[a_{j_3},n_t],P_2[m_{i},m_l],(a_1,\omega )]$ satisfies Lemma \ref{forb}$(1.a)$. Also notice that $m_{l-1} \leqslant_{T}\omega$, since else $[P_2[m_1,m_{i+1}],(a_{j_1},\omega ),P_1[a_{j_3},n_t]]$ satisfies Lemma \ref{forb}$(2)$. Now observe that $P_1\cap T]\omega , m_l[=\phi$, since else  $[(a_{j_2},\omega ),(n_{i_1},n_{i_1+1}),(m_{l-1},m_l)]$   satisfies Lemma \ref{forb}$(3)$, where $i_1$ is maximal satisfying $n_{i_1} \leqslant_{T}\omega$. Then $a_{j_4}=n_t=m_l$ and $n_{t-1} \leqslant_{T}\omega$. So $\mid$$N^-_C(\omega )$$\mid$ $= 4$. Now since $m_{l-1}\in T[r,a_1]$, then assume without loss of generality that $a_{j_5}=a_1\leqslant_{T}a_{j_1}=a_2$. So the union of $T[m_{l-1},a_{j_1}]\cup(a_{j_1},\omega )$, $(m_{l-1},m_l)$, $T[a_{j_2},a_{j_3}]\cup P_1[a_{j_3},m_l]$ and $(a_{j_2},\omega)$ is a $S$-$C(k,1,k,1)$ in $D$, a contradiction. $\hfill {\blacklozenge}$
	\begin{claim}\label{cl8}
		$C\notin C_4^1$.
	\end{claim}
	\noindent \sl {Subproof.} \upshape Assume to the contrary that $C\in C_4^1$. First observe that $\omega \notin T[r,n_1[\backslash \{x_1\}$, since else there exists $a_j\in (P_1\cup P_2)\backslash \{n_1\}$, such that $(\omega ,a_j)\in A(W)$. Due to symmetry, we will assume that $a_j\in P_2$. Then $[Q_1,(\omega ,a_j)]$ or $[(\omega ,a_j),Q_1]$ satisfies Lemma \ref{forb}$(4)$, a contradiction. As $\omega \notin T[r,n_1[\backslash \{x_1\}$, then by symmetry we may assume that $N_C(\omega )\subseteq P_1\cup \{x_1\}$.  Assume now that either  $\omega \in T_{n_t}\backslash \{n_t\}$, or $\omega$ and $n_t$ are not ancestors (clearly if the latter holds, then $\textrm{l.c.a}\{ \omega ,n_t\}\in T[n_4,n_t[$ as $W$ is a $5$-wheel). Then in both cases there exist at least three in-neighbors of $\omega$ in $P_1\backslash \{n_t\}$, say $a_{j_1}\leqslant_{T}a_{j_2}\leqslant_{T}a_{j_3}$, and so $[(a_{j_1},\omega ), (a_{j_2},\omega ), Q_1]$ satisfies  Lemma \ref{forb}$(1.a)$ or Lemma \ref{forb}$(1.b)$. So $\omega \in T]n_1,n_t[\backslash P_1$. Let $i$ be minimal satisfying $\omega \leqslant_{T}n_{i}$. If $\mid$$N^+_{P_1}(\omega )$$\mid$ $\geq 3$ or $\mid$$N^-_{C}(\omega )$$\mid$ $\geq 3$ with $n_i\neq n_t$ in case $\mid$$N^-_{C}(\omega )$$\mid$ $\geq 3$, then there exists $j\in [5]$ such that $a_j\in V(P_1)\backslash\{n_t,x_1,n_{i},n_{i-1}\} $, and $[(\omega ,a_j),Q_1,(n_{i-1},n_{i})]$, or $[(n_{i-1},n_{i}),Q_1,(a_j,\omega)]$ satisfies  Lemma \ref{forb}$(3)$, a contradiction. Then $\mid$$N^-_C(\omega )$$\mid$ $\geq 4$. Let $a_{i_1}\leqslant_{T}a_{i_2}\leqslant_{T}a_{i_3}\leqslant_{T}a_{i_4}$ be distinct in-neighbors of $\omega$ in $C$. So the union of $T[x_1,a_{i_2}]\cup (a_{i_2},\omega)$, $Q_1$, $T[a_{i_3},a_{i_4}]\cup P_1[a_{i_4},n_t]$ and $(a_{i_3},\omega)$ is a $S$-$C(k,1,k,1)$, a contradiction. $\hfill {\blacklozenge}$
	\begin{claim}\label{cl9}
		$C\notin C_4^3\cup C_4^4$.
	\end{claim}
	\noindent \sl {Subproof.} \upshape Assume to the contrary that $C\in C_4^3\cup C_4^4$.  Let $v=\textrm{l.c.a}\{ n_t,m_l\}$. Notice that $\omega \notin T[v,n_t[\cup T[v,m_l[$, since else $[Q_1,(m_{l-1},m_l)]$ or $[Q_2,(n_{t-1},n_t)]$ satisfies  Lemma \ref{forb}$(4)$, a contradiction. Also notice that for all $p\in T[r,v[\backslash\{n_{t-1}\}$, there exists no vertex $q\in T_{n_t}$ such that $(p,q)\in A(W)\backslash\{(x_1,n_t)\}$. Since else $[(y_1,q),(m_{l-1},m_l)]$ satisfies  Lemma \ref{forb}$(4)$, or $[(n_{t-1},n_t),(p,q),Q_2]$ satisfies  Lemma \ref{forb}$(1.b)$, or $[(m_{l-1},m_l),Q_2,(p,q)]$ satisfies  Lemma \ref{forb}$(1.b)$, a contradiction. Similarly we prove that for all $p\in T[r,v[\backslash\{m_{l-1}\}$, there exists no vertex $q\in T_{m_l}$ such that $(p,q)\in A(W)\backslash\{(x_1,m_l)\}$. So $\omega \notin T[v,n_t]\cup T[v,m_l]\cup T_{n_t}\cup T_{m_l}$, and if $\omega \in T[r,v[\backslash V(C)$, then $\{n_t,m_l\}\cap N_C(\omega )=\phi$. If $\omega$ and $n_t$ are not ancestors, then there exist two distinct in-neighbors $a_{j_1}, a_{j_2}$ of $\omega$ in $T[n_1,v[\cap V(C)$, and so $[(a_{j_1},\omega ),(a_{j_2},\omega ),Q_1]$  satisfies  Lemma \ref{forb}$(1.b)$, a contradiction. Then $\omega \in T[r,v[$. Assume that $\omega \leqslant_{T}m_{l-1}$. Clearly, if there exist two out-neighbors of $\omega$ in $P_1[n_2,n_{t-1}]$, say $a_{j_1}, a_{j_2}$, then $[(\omega ,a_{j_1}),(\omega ,a_{j_2}),(m_{l-1},m_l)]$ satisfies  Lemma \ref{forb}$(2)$. So $\mid$$N_{P_2[n_1,m_{l-1}]}(\omega)$$\mid$ $\geq 3$, and hence $C\in C_4^4$. Let $m_{i_1}\leqslant_{T}m_{i_2}\leqslant_{T}m_{i_3}$ be distinct neighbors of $\omega$ in $P_2[n_1,m_{l-1}]$. If $\omega \leqslant_{T}n_1$, then $[(\omega ,m_{i_2}),(\omega ,m_{i_3}),(n_1,n_2)]$ satisfies  Lemma \ref{forb}$(2)$, a contradiction. Then $\omega \in T]n_1,m_{l-1}[\backslash V(C)$. Let $i_4$ be maximal satisfying $m_{i_4} \leqslant_{T}\omega$. As $\mid$$N_{P_2[n_1,m_{l-1}]}(\omega)$$\mid$ $\geq 3$, then there exists $ j\in \{i_1,i_2,i_3\}$ such that $[(\omega ,m_j),Q_2,(m_{i_4},m_{i_4+1})]$ or  $[(m_{i_4},m_{i_4+1}),Q_2,$ $(m_j,\omega)]$ satisfies  Lemma \ref{forb}$(3)$, a contradiction. So $m_{l-1}\leqslant_{T}\omega$. Clearly, if $C\in C_4^4$, then there exists no $p\in T[y_1,m_{l-1}[\cap V(C)$ such that $p\in N_{C}(\omega )$, since else $[(n_1,n_2),(p,\omega ),(m_{l-1},m_l)]$ satisfies  Lemma \ref{forb}$(2)$, a contradiction. So $\mid$$N_{P_1[n_2,n_{t-1}]}(\omega)$$\mid$ $\geq 4$. Now similarly as in the above case we prove that if $\omega\leqslant_{T}n_{t-1}$, then Lemma \ref{forb}$(3)$ is satisfied, a contradiction. Hence $n_{t-1}\leqslant_{T}\omega$. Let $n_{i_1}\leqslant_{T}n_{i_2}\leqslant_{T}n_{i_3}\leqslant_{T}n_{i_4}$ be distinct in-neighbors of $\omega$ in $P_2[n_2,n_{t-1}]$, then the union of $T[y_1,n_{i_1}]\cup (n_{i_1},\omega )$, $Q_1$, $T[n_{i_2},n_{t-1}]\cup (n_{t-1},n_t)$ and $(n_{i_2},\omega )$ is a $S$-$C(k,1,k,1)$, a contradiction. This confirms our claim. $\hfill {\blacklozenge}$
	\begin{claim}\label{cl10}
		If $C\in (\bigcup_{i=1}^{3}C_6^i)\cup(\bigcup_{i=5}^{8}C_4^i)$, then $ \omega \in T]y_1,m_l[$. If $C\in C_6^4$, then $\omega \in T]x_1,m_2[$. And if $C\in C_8$, then $\omega \in T]x_1,m_2[\cup T]x_1,z_2[$.
	\end{claim}
	\noindent \sl {Subproof.} \upshape First we will show that if $C\in C_6^4\cup C_8$, then $\omega \notin T[r,x_1[\backslash \{y_1\}$, and if $C\in C_6^4$, then $\omega \notin T]x_1,z_2[$. Assume that $C\in C_6^4\cup C_8$ and $\omega \in T[r,x_1[\backslash \{y_1\}$, then there exists $a_\alpha \in V(C)\backslash \{y_1,x_1\}$, such that $(\omega ,a_\alpha )\in A(W)$ for some $\alpha\in [5]$, and so $[(\omega ,a_\alpha ),Q_j]$ or $[Q_j,(\omega ,a_\alpha )]$ satisfies  Lemma \ref{forb}$(4)$ with $j\in\{ 2,3\}$, a contradiction. Assume now that $C\in C_6^4$ and $\omega \in T]x_1,z_2[$, and let $j_1$ be minimal satisfying $\omega\leqslant_{T}w_{j_1}$. If either $\mid$$N^-_{Q_4\cup \{y_1\}}(\omega )$$\mid$ $\geq 3$ with $w_{j_1}\neq z_2$, or $\mid$$N^+_{Q_4}(\omega )$$\mid$ $\geq 3$, then there exists $\alpha\in [5]$ such that $[(\omega ,a_{\alpha}),Q_3,(w_{j_1-1},w_{j_1})]$ or $[(w_{j_1-1},w_{j_1}),Q_3,(a_{\alpha},\omega )]$ satisfies  Lemma \ref{forb}$(3)$, with $a_{\alpha}\in N_{Q_4\backslash \{z_2,w_{j_1-1},w_{j_1}\}}(\omega )$, a contradiction. Then $\mid$$N^-_{Q_4\cup \{y_1\}}(\omega )$$\mid$ $\geq 4$, and so the union of $T[y_1,a_{i_2}]\cup (a_{i_2},\omega)$, $Q_3$, $T[a_{i_3},a_{i_4}]\cup Q_4[a_{i_4},z_2]$ and $(a_{i_3},\omega)$ is a $S$-$C(k,1,k,1)$, where $a_{i_1}\leqslant_{T}a_{i_2}\leqslant_{T}a_{i_3}\leqslant_{T}a_{i_4}$ are distinct in-neighbors of $\omega$ in $Q_4\cup \{y_1\}$, a contradiction. So if $C\in C_6^4$, then $\omega \notin T]x_1,z_2[$. Let's assume now that $C\in (\bigcup_{i=1}^{4}C_6^i)\cup(\bigcup_{i=5}^{8}C_4^i)\cup C_8$. Moreover assume that $ \omega \notin T]y_1,m_l[$ in case $C\in (\bigcup_{i=1}^{3}C_6^i)\cup(\bigcup_{i=5}^{8}C_4^i)$, $\omega \notin T]x_1,m_2[$ in case $C\in C_6^4$, and $\omega \notin T]x_1,m_2[\cup T]x_1,z_2[$ in case $C\in C_8$. Then the above observations with our assumption implies that $\mid$$T]y_1,m_l[\cap N_C(\omega )$$\mid$ $\geq 3$ or $\mid$$T[x_1,z_2[\cap N_C(\omega )$$\mid$ $\geq 3$. Let $a_{i_j}\in V(C)\backslash \{y_1,m_l,z_2\}$ for $j=1,2$, such that $a_{i_1}\leqslant_{T}a_{i_2}$ be two distinct neighbors of $\omega$. If $l_{T}(\omega )>l_{T}(y_1)$, then $[(a_{i_1},\omega ),(a_{i_2},\omega ),Q_j]$ satisfies  Lemma \ref{forb}$(1.a)$ or Lemma \ref{forb}$(1.b)$ with $j\in \{2,3\}$, a contradiction. Then $\omega \leqslant_{T}y_1$, and so $[(\omega ,a_{i_1}),(\omega ,a_{i_2}),Q_j]$ satisfies  Lemma \ref{forb}$(2)$ with $j\in \{2,3\}$, a contradiction. $\hfill {\blacklozenge}$
	\begin{claim}\label{cl12}
		$C\notin C_6^1\cup C_6^2$.
	\end{claim}
	\noindent \sl {Subproof.} \upshape Assume the contrary is true, then Claim \ref{cl10} implies that $ \omega \in T]y_1,m_2[$. As $W$ is a $5$-wheel, then either $\mid$$N^+_C(\omega )$$\mid$ $\geq 3$ or $\mid$$N^-_C(\omega )$$\mid$ $\geq 3$. Assume first that $\mid$$N^+_C(\omega )$$\mid$ $\geq 3$. Let $a_{i_1}\leqslant_{T}a_{i_2}\leqslant_{T}a_{i_3}$ be three out-neighbors of $\omega$ in $C$, let $p\in V(C)$ such that $T]\omega ,p[\cap C=\phi$, and let $q\in N^-_C(p)$ (if exist). Clearly, $\omega \leqslant_{T}n_t$. Assume now that $\omega \leqslant_{T}z_2$. If $a_{i_3}\in T]z_2,m_2]$, then $[(\omega ,a_{i_3}),Q_1,Q_3]$ satisfies  Lemma \ref{forb}$(1.a)$, a contradiction. Then $a_{i_2}\in T]\omega ,z_2[$, and so $[(x_1,x_2),Q_2,(\omega ,a_{i_2})]$ or $\theta =[(\omega ,a_{i_2}),Q_2,(q,p)]$ satisfies  Lemma \ref{forb}$(3)$, a contradiction. If $\omega \in T]n_1,n_t[$ or $\omega \in T]z_2,x_{t_1-1}[$ (note that in case $C\in C_6^2$, we may have: $\omega \in T]z_2,x_{t_1-1}[$), then $\theta$ satisfies  Lemma \ref{forb}$(3)$, a contradiction. So if $C\in C_6^1$ (resp. $C\in C_6^2$),  then $\omega \in T]z_2,n_1[$ (resp. $\omega\in T]x_{t_1-1},n_1[$). Then the union of $T[y_1,x_1]\cup Q_1$, $Q_2$, $T[\omega ,n_1]\cup P_2$ and $(\omega ,a_{i_2})\cup T[a_{i_2},n_t]$ is a $S$-$C(k,1,k,1)$ in $D$, a contradiction. Hence, $\mid$$N^-_C(\omega )$$\mid$ $\geq 3$. Clearly $\omega \notin T]y_1,x_1[$. Let $a_{i_1}\leqslant_{T}a_{i_2}\leqslant_{T}a_{i_3}$ be three in-neighbors of $\omega$ in $C$, let $p\in V(C)$ such that $T]p,\omega [\cap C=\phi$, and let $q\in N^+_C(p)$ (if exist). Assume first that $n_1\leqslant_{T}\omega$. If $a_{i_1}\in T[y_1,n_1[$, then $[(a_{i_1},\omega ),Q_1,P_2]$ satisfies  Lemma \ref{forb}$(2)$, a contradiction. Then $a_{i_2}\in T]n_1,n_t[$, and so $[(a_{i_2},\omega ),Q_2, (x_{t_1-1},n_t)]$ or $\theta_{1}=[(p,q),Q_2,(a_{i_2},\omega )]$ satisfies  Lemma \ref{forb}$(3)$, a contradiction. Now if $\omega \in T]x_1,z_2[$ or $\omega \in T]x_2,n_1[$ (in case $C\in C_6^2$, we may have: $\omega \in T]x_2,n_1[$), then $\theta_{1}$ satisfies  Lemma \ref{forb}$(3)$, a contradiction. So if $C\in C_6^1$ (resp. $C\in C_6^2$),  then $\omega \in T]z_2,n_1[$ (resp. $\omega\in T]z_2,x_2[$), and hence the union of $Q_1\cup T[n_t,m_2]$, $T[x_1,a_{i_2}]\cup (a_{i_2},\omega )$, $Q_3\cup T[z_2,\omega ]$ and $Q_2$ is a $S$-$C(k,1,k,1)$, a contradiction. This completes the proof. $\hfill {\blacklozenge}$
	\begin{claim}\label{cl11}
		$C\notin C_6^3\cup C_6^4\cup C_8\cup (\bigcup_{j=5}^{8}C_4^j)$.
	\end{claim}
	\noindent \sl {Subproof.} \upshape Assume the contrary is true. Then Claim \ref{cl10} implies that if $C\notin C_6^4\cup C_8$, then $\omega \in T]y_1,m_l[$, if $C\in C_6^4$, then $\omega \in T]x_1,m_2[$, and if $C\in C_8$ then $\omega \in T]x_1,m_2[\cup T]x_1,z_2[$. By symmetry, if $C\in C_8$, then we will assume that $\omega \in T]x_1,m_2[$. Notice that if $C\in C_6^3$, then $z_2\notin N_C(\omega )$ since otherwise $[(\omega ,z_2),(x_1,z_2),Q_2]$ satisfies  Lemma \ref{forb}$(1.b)$. We will prove a useful observation before taking all the possible positions of $\omega$: For all $p\in T]y_1,n_1[\backslash V(C)$, there exists no $q\in T]n_1,m_l]\backslash \{n_t\}$ such that $(p,q)\in A(W)$. Assume else and notice that in case $C\in C_6^4\cup C_8$, then clearly Claim \ref{cl10} implies that $p\notin T]y_1,x_1[$. If $C\in C_4^5\cup C_4^6\cup C_6^3\cup C_6^4\cup C_8$, then $[(p,q),Q_1,P_2]$ satisfies  Lemma \ref{forb}$(2)$ or $[(p,q),(n_1,m_2),Q_1]$ satisfies  Lemma \ref{forb}$(1.a)$ or $[(p,m_2),P_2,Q_3]$ satisfies  Lemma \ref{forb}$(1.b)$, a contradiction. And if $C\in C_4^7\cup C_4^8$, then $[P_1,Q_2,(p,q)]$ satisfies  Lemma \ref{forb}$(3)$ or $[(p,q),(m_{l-1},m_l),Q_1]$ satisfies  Lemma \ref{forb}$(1.a)$, a contradiction. This confirms our observation. Now we will discuss according to the position of $\omega$. Assume first that $\omega \leqslant_{T}n_1$. Then our  observations with the fact that $W$ is a $5$-wheel implies that $C\in C_4^5\cup C_6^4\cup C_8$, and $\mid$$N^+_{C}(\omega )$$\mid$ $\geq 3$ or $\mid$$N^-_{C}(\omega )$$\mid$ $\geq 3$. Assume that $\mid$$N^+_{C}(\omega )$$\mid$ $\geq 3$, and let $a_{i_1}\leqslant_{T}a_{i_2}\leqslant_{T}a_{i_3}$ be three out-neighbors of $\omega$ in $C$. Clearly our observation implies that  $\omega \leqslant_{T} x_{t_1-1}$, $a_{i_1}\in Q_1$, and $a_{i_3}\neq m_2$. So the union of $T[y_1,x_1]\cup Q_1[x_1,a_{i_1}]\cup T[a_{i_1},a_{i_2}]$, $Q_2$, $(\omega ,a_{i_3})\cup T[a_{i_3},m_2]$ and $(\omega ,a_{i_2})$ is a $S$-$C(k,1,k,1)$ in $D$, a contradiction. Then $\mid$$N^-_{C}(\omega )$$\mid$ $\geq 3$, and so  $[(x_{i},x_{i+1}),Q_2,(a_{i_2},\omega )]$ satisfies  Lemma \ref{forb}$(3)$, where $a_{i_1}\leqslant_{T}a_{i_2}\leqslant_{T}a_{i_3}$ are three in-neighbors of $\omega$ in $C$, and $i$ is maximal satisfying $x_{i}\leqslant_{T}\omega$.  Hence $n_1 \leqslant_{T}\omega \leqslant_{T}m_l$. Clearly if there exists $p\in N_C(\omega )\cap T[y_1,n_1[$, then $[Q_1,(p,\omega ),P_2]$ satisfies  Lemma \ref{forb}$(2)$, a contradiction. Then $ N_C(\omega )\cap T[y_1,n_1[=\phi$. We will notice one more observation: For all $p\in T]n_1,n_t[$, there exists no $q\in T]n_t,m_l]$ such that $(p,q)\in A(W)$, since otherwise $[(p,q),P_2,Q_1]$ satisfies  Lemma \ref{forb}$(1.a)$ or $[(p,q),P_1,Q_2]$ satisfies  Lemma \ref{forb}$(3)$, a contradiction. Assume that $\omega \in T]n_t,m_l[$, then our observation with the fact that $W$ is a $5$-wheel implies that $C\in C_4^6$ and $\mid$$N_{P_2[m_2,m_l]}(\omega )$$\mid$ $\geq 3$. If $\omega \in T]n_t,m_2[$, then our observations implies that $[(\omega ,a_{j_2}),Q_2,(n_1,m_2)]$  satisfies  Lemma \ref{forb}$(3)$, where $a_{j_1}\leqslant_{T}a_{j_2}\leqslant_{T}a_{j_3}$ are three out-neighbors of $\omega$ in $P_2[m_2,m_l]$, a contradiction. So $\omega\in T]m_2,m_l[$. Then $N^-_C(\omega )\cap \{n_1,n_t\}=\phi$, since else $[(n_t,\omega ),Q_2,(n_1,m_2)]$ satisfies  Lemma \ref{forb}$(3)$, or the union of $Q_1\cup T[n_t,m_2]$, $Q_2$, $(n_1,\omega )\cup T[\omega ,m_l]$, $(n_1,m_2)$ is a $S$-$C(k,1,k,1)$ in $D$, a contradiction. Hence $N_C(\omega )\subseteq V(P_2)\backslash \{n_1\}$, and so there exists $\alpha \in [5]$ such that $[(\omega ,a_{\alpha}),Q_2,(m_{i},m_{i+1})]$ satisfies  Lemma \ref{forb}$(3)$, where $i$ is maximal satisfying $m_{i}\leqslant_{T}\omega$, or the union of $T[x_1,a_{j_1}]\cup (a_{j_1},\omega )$, $Q_2$, $T[a_{j_2},a_{j_3}]\cup P_2[a_{j_3},m_l]$ and $(a_{j_2},\omega )$ is a $S$-$C(k,1,k,1)$ in $D$, where $a_{j_1}\leqslant_{T}a_{j_2}\leqslant_{T}a_{j_3}$ are three in-neighbors of $\omega$ in $P_2[m_2,m_{l-1}]$,  a contradiction. So $\omega \in T]n_1,n_t[$, and hence the above observations implies that $N_C(\omega )\subseteq T[n_1,n_t]$. Assume first that $\mid$$N^+_{C}(\omega )$$\mid$ $\geq 3$, and let $a_{i_1}\leqslant_{T}a_{i_2}\leqslant_{T}a_{i_3}$ be three out-neighbors of $\omega$ in $C$. Let $p\in V(C)$ such that $T]\omega ,p[\cap C=\phi$, and let $q\in N^-_C(p)$. Then either $C \notin C_4^8$ and so $[(\omega ,a_{i_2}),Q_2,(q,p)]$ satisfies  Lemma \ref{forb}$(3)$, or $C \in C_4^8$ and so $[(\omega ,a_{i_3}),P_2,Q_1[x_1,p]]$ satisfies  Lemma \ref{forb}$(1.a)$, a contradiction. Then $\mid$$N^-_{C}(\omega )$$\mid$ $\geq 3$. Let $a_{i_1}\leqslant_{T}a_{i_2}\leqslant_{T}a_{i_3}$ be three in-neighbors of $\omega$ in $C$. Let $p\in V(C)$ such that $T]p,\omega [\cap C=\phi$, and let $q\in N^+_C(p)$. Then either $C \notin C_4^7$ and so $[(p,q),Q_2,(a_{i_2},\omega)]$ satisfies  Lemma \ref{forb}$(3)$, or $C \in C_4^7$ and so $[(a_{i_1},\omega),P_1,P_2[p,m_l]]$ satisfies  Lemma \ref{forb}$(2)$, a contradiction. This completes the proof. $\hfill {\blacklozenge}$ \medbreak 
	\noindent
	All the above discussion implies that $C\notin \mathcal{C}$, a contradiction. This completes the proof. $\hfill {\square}$
	\end{proof}

\subsection{Coloring $D_i^2$}
In this section, we study the chromatic number of $D_i^2$. In fact, the coloring of $D_i^2$  heavily depends on the following  observation:
\begin{lemma}\label{acyclic}
	Let $D$ be an acyclic digraph. Then $G(D)$ is $\Delta^+(D)$-degenerate and thus $\chi(D) \leq \Delta^+(D)+1$. 
\end{lemma}
\begin{proof}
	Let $G$ be a subgraph of $G(D)$ and let $H$ be the subdigraph of $D$ whose underlying graph is $G$.  Let $P$ be a longest directed path of $H$. One may easily see that the initial end of $P$, say $u$,  has no in-neighbors in $H$, since otherwise we get either a directed path longer than $P$  or a directed cycle in $H$. These are contradictions to the facts that $P$ is a longest directed path of $H$ and that $D$ is acyclic. Hence, the only neighbors of $u$ in $G$ are its out-neighbors in $H$.  This implies the desired result. $\hfill {\square}$
\end{proof}
\begin{proposition}\label{p2}
	$\chi(D_i^2) \leq 6$  for all $i \in [2k]$.
\end{proposition}
\begin{proof}
	Let $B_1$ and $B_2$ be a partition of the vertex-set of  $D_i^2$, with $B_1:=\{v \in V_i|d^+_{D_i^2}(v) \leq 1\}$ and $B_2:=V_i \setminus B_1$. Obviously, $\Delta^+(D_i^2[B_1]) \leq 1$.  Now we are  going to prove that $\Delta^+(D_i^2[B_2]) \leq 3$. Assume the contrary is true and let $u$ be a vertex of $B_2$ whose out-degree in $D_i^2[B_2]$  is at least $4$.  By the definition of $A_2$, it is easy to see that  all the out-neighbors of $u$  belong to $T[r,u]$. This induces an ordering of the  out-neighbors of $u$ in $D_i^2[B_2]$ with respect to $\leqslant_{T}$, say $v_1, v_2, ...,v_t$ with $v_{i-1} \leqslant_{T} v_{i}$ for all $2\leq i \leq t$. According to our assumption, note that  $t$ must be greater than 3.  Moreover, the definition of $B_2$ forces the existence of an out-neighbor $w_i$ of $v_i$ other than $v_1$, for each  $ 2 \leq i \leq t-1$.   Due to the definition of $A_2$,  $w_i$ and $v_1$ must be  ancestors.  More precisely,  $w_i  \leqslant_{T} v_1$ for all  $ 2 \leq i \leq t-1$, since otherwise if there exists $i_0 \in \{2, ..., t-1\}$ such that $v_1 \leqslant_{T} w_{i_0}$, then  the union of $T[v_{i_0}, v_t], (v_{i_0}, w_{i_0}), (u,v_1) \cup T[v_1,w_{i_0}]$ and $ (u,v_t)$ is a $S$-$C(k,1,k,1)$ in $D$, a contradiction. To reach the final contradiction, we  consider two out-neighbors $v_i,v_j$  of $u$ with  $ 2 \leq i <j \leq t-1$ and their respective out-neighbors $w_i, w_j$.  Note that the existence of $v_i$ and $v_j$ is guaranteed by the assumption that $t \geq 4$. Moreover, note  that possibly $w_i=w_j$. In view of the above observation,  $w_i  \leqslant_{T} v_1$ and $w_j  \leqslant_{T} v_1$.  If $w_i \leqslant_{T} w_j$, then the union of $T[v_j,v_t], (v_j,w_j), (u,v_1) \cup T[v_1, v_i] \cup (v_i,w_i) \cup T[w_i,w_j]$ and $ (u,v_t)$ forms a $S$-$C(k,1,k,1)$ in $D$, a contradiction. Otherwise, the union of $T[v_j,v_t], (v_j,w_j) \cup  T[w_j,w_i] , (u,v_1) \cup T[v_1, v_i] \cup (v_i,w_i)$ and $ (u,v_t)$ is a $S$-$C(k,1,k,1)$ in $D$, a contradiction. This proves that $\Delta^+(D_i^2[B_2]) \leq 3$. Consequently, due to the fact that $D_i^2$ is acyclic together with Lemma \ref{acyclic}, it follows that $D_i^2[B_1]$ is $2$-colorable and    $D_i^2[B_2]$ is  $4$-colorable. Therefore, by assigning the vertices of $B_1$ 2 colors and those of  $B_2$ 4 new colors,  we get a proper $6$-coloring of $D_i^2$. This completes the proof. $\hfill {\square}$
\end{proof}

 \subsection{Coloring $D_i^3$}
This section is devoted to color $D_i^3$ properly. 
  \begin{proposition}\label{p3}
 	$\chi(D^3_i) \leq 4k+2$  for all $i \in [2k]$.
 \end{proposition}
 \begin{proof}
Assume to the contrary that $\chi(D_{i}^{3}) \geq 4k+3$. Due to Theorem \ref{path}, $D_{i}^{3}$ contains a copy $Q$ of $P(2k+1,2k+1)$, which is the union of two directed paths $Q_{1}$ and $Q_{2}$ which are disjoint except in their initial vertex, say $Q_{1}=y_{0},y_{1},..., y_{2k}$ and $Q_{2}=z_{0},z_{1},..., z_{2k}$ with $y_0=z_0$. We need to prove a series of assertions as follows:
\begin{assertion}\label{assertion1}
For all $ i \in [2k-1]$ and  $ j \in [2k]$,	$y_i$ is not an ancestor of $z_j$ and $z_i$ is not an ancestor of $y_j$.
\end{assertion}
\noindent \sl {Proof of Assertion \ref{assertion1}. }\upshape Due to symmetry, we are going to show that $y_i$ is not an ancestor of $z_j$ for all $1 \leq i \leq 2k-1$ and  $1 \leq j \leq 2k$.  Assume the contrary is true.  Then there exists $ i \in [2k-1]$  such that $y_i \leqslant_{T} z_j$ for some $j \in [2k]$. Suppose that $y_i$ and  $z_j$ are chosen so that $T[y_i,z_j] \cap Q_2[y_0,z_j] =\{z_j\}$. By the definition of $A_3$, note that $y_{i+1} \notin T[y_i,z_j]$, as $(y_i,y_{i+1}) \in A_3$.  Observe now that $T[r,y_{i+1}] \cap (Q_1[y_0,y_i] \cup Q_2[y_0, z_j]) \neq \emptyset$, since otherwise the union of $T[\beta, y_{i+1}], T[\beta, y_0] \cup Q_2[y_0,z_j], T[y_i,z_j] $ and $(y_i,y_{i+1})$ forms a $S$-$C(k,1,k,1)$ in $D$, with $\beta=\textrm{l.c.a}\{y_0, y_{i+1}\}$. This is a contradiction to the fact that $D$ is $C(k,1,k,1)$-subdivision-free. Let $\alpha \in T[r,y_{i+1}]$ such that $T[\alpha, y_{i+1}] \cap (Q_1[y_0,y_i] \cup Q_2[y_0, z_j]) =\{ \alpha\}$. Clearly, $ \alpha \notin \{z_j,y_i\}$.  If $\alpha \in Q_1$, then the union of $Q_1[y_0, \alpha] \cup T[\alpha, y_{i+1}], Q_2[y_0, z_j], T[y_i,z_j]$   and $(y_i,y_{i+1})$  is a $S$-$C(k,1,k,1)$ in $D$, a contradiction. This implies that $\alpha$ must belong to $Q_2 -y_0$.  But the union of $T[\alpha, y_{i+1}], Q_2[\alpha, z_j], T[y_i, z_j]$ and $(y_i, y_{i+1})$ is a $S$-$C(k,1,k,1)$ in $D$, a contradiction. This confirms our assertion.  $\hfill {\lozenge}$ \medbreak 
\noindent In what follows, we denote by    $x_1=\textrm{l.c.a}\{y_0,y_1\}$,  $x_2=\textrm{l.c.a}\{y_0, z_1 \}$ and  $x_3=\textrm{l.c.a}\{y_1, z_1\}$.
\begin{assertion}\label{assertion2}
$x_3 \notin \{x_1,x_2\}$. 
\end{assertion}
\noindent \sl {Proof of Assertion \ref{assertion2}. }\upshape Suppose the contrary is true, that is, $x_3=x_1$ or $x_3=x_2$. Without loss of generality, assume that $x_3=x_2$. This means that $x_2 \leqslant_{T} x_1$. By the definition of $D^3_i$, note that $T[x_1, y_1]$ and $T[x_2,z_1]$ are of length at least $2k$. Throughout the proof of this assertion, we denote by $T_1=T[x_1,y_1] \cup T[x_2,z_1] \cup T[x_2,y_0]$.
\begin{claim}\label{cl1}
	$y_0$ is not an ancestor neither of $y_i$ nor of $z_i$ for all $ i \in [2k]$.  
\end{claim}
\noindent \sl {Subproof. }\upshape Due to symmetry, we are going to prove that $y_0$ is not an ancestor of $y_i$ for all $i \in [2k]$. Assume the contrary is true. Then there exists $i \in [2k]$ such that $y_0 \leqslant_{T} y_i$. Assume that $y_i$ is chosen so that $y_0$ is not an ancestor of $y_j$ for all $j <i$. Clearly, $i>1$. Then the union of $T[x_2,z_1], T[x_2,y_1] \cup Q_1[y_1,y_i], T[y_0,y_i]$ and $(y_0,z_1)$ is a $S$-$C(k,1,k,1)$, a contradiction. This affirms our claim. $\hfill {\blacklozenge}$
\begin{claim}\label{cl2}
	For all $0 \leq i \leq k$ and $j \in [2k]$, $z_i$ is not an ancestor of $z_j$. 
\end{claim}
\noindent \sl {Subproof. }\upshape We proceed by induction on $i$. The base case $i=0$ follows by Claim \ref{cl1}. Suppose now that $z_t$ is not an ancestor of $z_j$ for all $ 0 \leq t <i$ and $j \in [2k]$. Our aim is to prove that $z_i$ is not an ancestor of $z_j$ for all  $j \in [2k]$. Assume the contrary is true, that is, there exists $j \in [2k]$ such that $z_i \leqslant_{T} z_j$. Assume that $z_j$ is chosen so that $l_{T}(z_j)$ is maximal and $T[z_i,z_j] \cap Q_2=\{z_i,z_j\}$. Clearly, $z_{i+1} \notin T[z_i,z_j]$.  Let $\alpha_1$ be the vertex of $T[r,z_{i+1}]$ such that $T[\alpha_1,z_{i+1}] \cap T_1=\{\alpha_1\}$ if $x_2 \in T[r,z_{i+1}]$ and $\alpha_1=\textrm{l.c.a}\{x_2, z_{i+1}\}$  otherwise. Note that $T[\alpha_1, z_{i+1}] \cap Q_2[y_0,z_i] = \emptyset$, due to the induction hypothesis. Moreover, $T[\alpha_1, z_{i+1}] \cap Q_1[y_1,y_{2k-1}] = \emptyset$, according to Assertion \ref{assertion1}. Now we are going to show that $\alpha_1 \in T]x_1,y_1]$. In fact, if $\alpha_1 \notin T]x_1,y_1]$, we consider two possibilities: If $\alpha_1 \in T[x_2, z_1]$, then the union of $T[x_2, y_1], T[x_2, z_{i+1}] \cup Q_2[z_{i+1},z_j],  Q_2[y_0,z_i] \cup T[z_i, z_j]$ and $(y_0,y_1)$ is a $S$-$C(k,1,k,1)$, a contradiction. Else if $\alpha_1 \notin T[x_2, z_1]$, then $\alpha_1$ and $x_1$ are ancestors and so   the union of $T[\beta, y_1], T[\beta, z_{i+1}] \cup  Q_2[z_{i+1},z_j],  Q_2[y_0,z_i] \cup T[z_i, z_j]$ and $(y_0,y_1)$ is a $S$-$C(k,1,k,1)$ in $D$ with $\beta=\textrm{min}_T\{x_1, \alpha_1\}$, a contradiction. Now we consider the vertex  $\alpha_2$ of  $T[r,y_2]$ such that $T[\alpha_2,y_2] \cap (T_1 \cup T[\alpha_1,z_{i+1}] \cup T[z_i,z_j])=\{\alpha_2\}$ if $x_2 \in T[r,y_{2}]$ and $\alpha_2=\textrm{l.c.a}\{x_2, y_2\}$  otherwise. Note that $T[\alpha_2, y_2] \cap Q_2[y_0,z_{2k-1}] =\emptyset$, according to Assertion \ref{assertion1} and Claim \ref{cl1}. Moreover, $T[\alpha_2, y_2] \cap T[z_i,z_j] =\emptyset$, according to Assertion \ref{assertion1}. Actually,  $\alpha_2 \in T]\alpha_1, z_{i+1}]$.  If not, then  the union of $T[\beta, y_2]$, $T[\beta, z_{i+1}] \cup Q_2[z_{i+1},z_j], Q_2[y_0,z_i] \cup T[z_i,z_j]$ and $Q_1[y_0,y_2]$ is a $S$-$C(k,1,k,1)$ in $D$, where  $\beta=\textrm{min}_T\{\alpha_1,\alpha_2\}$ if $\alpha_1$ and $\alpha_2$ are ancestors and $\beta=\textrm{l.c.a}\{\alpha_1, \alpha_2\}$ otherwise. This is a contradiction to the fact that $D$ is $C(k,1,k,1)$-subdivision-free digraph. $*\{$Note here that $l(T[\alpha_2,y_2])<k$ and $l(T[\alpha_2,z_{i+1}])<k$, since otherwise the union of $T[\alpha_2,y_2], T[\alpha_2,z_{i+1}] \cup Q_2[z_{i+1},z_{j}], T[x_2,z_1] \cup Q_2[z_1,z_i] \cup T[z_i,z_j]$ and $T[x_2,y_1] \cup (y_1,y_2)$ is a $S$-$C(k,1,k,1)$, a contradiction. This implies that $l(T[\alpha_1,\alpha_2]) \geq k$, as $l(T[\alpha_1,y_2])\geq 2k$. Moreover, note that $l(Q_2[z_{i+1}, z_j]) \leq k-2$, since otherwise the union of $ T[\alpha_2,z_{i+1}] \cup Q_2[z_{i+1},z_{j}], T[\alpha_2,y_2],  T[x_2,y_1] \cup (y_1,y_2)$ and $T[x_2,z_1] \cup Q_2[z_1,z_i] \cup T[z_i,z_j]$ is a $S$-$C(k,1,k,1)$, a contradiction. This guarantees the existence of $z_{j+1}$. Let  $\alpha_3$ be the vertex of $T[r,z_{j+1}]$ such that $T[\alpha_3,z_{j+1}] \cap (T_1 \cup T[\alpha_1,z_{i+1}] \cup T[\alpha_2,y_2] \cup T[z_i,z_j])=\{\alpha_3\}$ if $x_2 \in T[r,z_{j+1}]$ and $\alpha_3$ is the l.c.a of $x_2$ and $z_{j+1}$ otherwise. Observe that $T[\alpha_3, z_{j+1}] \cap Q_2[y_0,z_j] =\emptyset$. In fact, $T[\alpha_3, z_{j+1}] \cap Q_2[y_0,z_{i-1}] =\emptyset$, due to the induction hypothesis. Moreover, $T[\alpha_3, z_{j+1}] \cap Q_2[z_{i+1},z_{j-1}] =\emptyset$, since otherwise if there exists $i+1 \leq t \leq j-1$ such that $z_{t}\leqslant_{T} z_{j+1}$, then the union of $T[\alpha_2,z_{i+1}] \cup Q_2[z_{i+1},z_t] \cup T[z_t, z_{j+1}], T[\alpha_2,y_2], T[x_2,y_1] \cup (y_1,y_2)$ and $T[x_2,z_{1}] \cup Q_2[z_1,z_i] \cup T[z_i,z_j] \cup (z_j,z_{j+1})$  is a $S$-$C(k,1,k,1)$, a contradiction. This  together with the maximality of $z_j$ imply that $z_i$ is also not an ancestor of $z_{j+1}$. To reach the final contradiction, we study the possible positions of $\alpha_3$. If $\alpha_3 \in T]\alpha_1, z_{i+1}]$, then the union of $T[\beta,z_{j+1}], T[\beta,y_2], T[x_2,y_1] \cup (y_1,y_2)$ and $T[x_2,z_1] \cup Q_2[z_1,z_i] \cup T[z_i,z_j] \cup (z_j,z_{j+1})$  with $\beta=\textrm{min}_T\{\alpha_2,\alpha_3\}$ is a $S$-$C(k,1,k,1)$, a contradiction. If $\alpha_3 \in T]\alpha_2,y_2]$, then $l(T[\alpha_2,z_{j+1}]) \geq 2k$.  But $l(T[\alpha_2,z_{j+1}]) =l(T[\alpha_2, \alpha_3])+l(T[\alpha_3,z_{j+1}])$ and $l(T[\alpha_2,\alpha_3]) <l(T[\alpha_2,y_2])<k$, then $l(T[\alpha_3,z_{j+1}]) \geq k$. Consequently, the union of $T[\alpha_3,z_{j+1}], T[\alpha_3,y_2], T[x_2,y_1] \cup (y_1,y_2)$ and $T[x_2,z_1] \cup Q_2[z_1,z_i] \cup T[z_i,z_j] \cup (z_j,z_{j+1})$ is a $S$-$C(k,1,k,1)$, a contradiction. Otherwise, let $\beta=\textrm{min}_T\{\alpha_1,\alpha_3\}$ if $\alpha_1$ and $\alpha_3$ are ancestors and $\beta=\textrm{l.c.a}\{\alpha_1, \alpha_3 \}$  otherwise. Then the union of $T[\beta,z_{i+1}], T[\beta,z_{j+1}], T[z_i,z_{j}] \cup (z_j,z_{j+1})$ and $(z_i,z_{i+1})$ is a $S$-$C(k,1,k,1)$ in $D$, a contradiction$\}*$ affirming that $z_i$ is not an ancestor of $z_j$ for all  $j \in [2k]$. This ends the proof.  $\hfill {\blacklozenge}$ \medbreak
  \noindent In a similar way, we can prove that $y_i$ is not an ancestor of $y_j$, for all $0 \leq i \leq k$ and $j \in [2k]$.  To complete the proof, we consider the vertices $\alpha_1$ and  $\alpha_2$ of $T[r,z_k]$ and $T[r,y_k]$ respectively such that $T[\alpha_1,z_{k}] \cap T_1 =\{\alpha_1\}$ if $x_2 \in T[r,z_k]$ and $\alpha_1=\textrm{l.c.a}\{x_2, z_k\}$ otherwise, and  $T[\alpha_2,y_{k}] \cap T_1 =\{\alpha_2\}$ if $x_2 \in T[r,y_k]$ and $\alpha_2=\textrm{l.c.a}\{x_2, y_k\}$ otherwise. Note that $(T[\alpha_1,z_k] \cup T[\alpha_2,y_k]) \cap (Q_1[y_0, y_{2k-1}] \cup Q_2[y_0,z_{2k-1}]) =\emptyset.$ Let $\beta_1=\textrm{min}_T\{\alpha_1,\alpha_2\}$ if $\alpha_1$ and $\alpha_2$ are ancestors and $\beta_1=\textrm{l.c.a}\{\alpha_1, \alpha_2\}$  otherwise. Given that $\beta_2=\textrm{l.c.a}\{y_k, z_k\}$, we study two cases: If $\beta_1=\beta_2$, then at least one of $T[\beta_1,y_k]$ and $T[\beta_1,z_k]$ has length greater than $k$. Consequently, the union of $T[\beta_1,z_k], T[\beta_1,y_k], Q_1[y_0,y_k]$ and $Q_2[y_0,z_k]$ is a $S$-$C(k,1,k,1)$ in $D$, a contradiction. Otherwise, we get $\alpha_1=\alpha_2=\beta_1$ and $\beta_2 \in T]\alpha_1,z_k]$. Clearly, $l(T[\beta_2,y_k])<k$ and $l(T[\beta_2,z_k])<k$, since otherwise the union of $T[\beta_2,z_k], T[\beta_2,y_k], Q_1[y_0,y_k]$ and $Q_2[y_0,z_k]$ would be  a $S$-$C(k,1,k,1)$ in $D$, a contradiction. To reach the final contradiction, let $\alpha_3$ be the vertex of $T[r,z_{k+1}]$  such that $T[\alpha_3,z_{k+1}] \cap (T_1 \cup T[\alpha_1, y_k] \cup T[\alpha_1,z_k]) =\{\alpha_3\}$ if $x_2 \in T[r,z_{k+1}]$ and $\alpha_3=\textrm{l.c.a}\{x_2, z_{k+1}\}$ otherwise. If $\alpha_3 \in T[\beta_2,z_k]$, then the union of $T[\beta_2,z_{k+1}], T[\beta_2, y_k], Q_1[y_0,y_k]$ and $Q_2[y_0,z_{k+1}]$ forms a $S$-$C(k,1,k,1)$ in $D$, a contradiction. Else if $\alpha_3 \in T]\beta_2,y_k]$, then $l(T[\alpha_3,z_{k+1}]) \geq k$, since otherwise $l(T[\beta_2,y_k]) > T[\beta_2,\alpha_3]\geq k$, a contradiction. Thus the union of $T[\alpha_3,z_{k+1}], T[\alpha_3, y_k], Q_1[y_0,y_k]$ and $Q_2[y_0,z_{k+1}]$ forms a $S$-$C(k,1,k,1)$ in $D$, a contradiction. Else if $\alpha_3 \notin T[\beta_2,y_k] \cup T[\beta_2,z_k]$, then consider $\beta_3=\{\alpha_3\}$ if $\alpha_3 \leqslant_{T} \beta_2$ and $\beta_3=\textrm{l.c.a}\{\alpha_3, \beta_2\}$ otherwise.  It is easy to check that $\beta_3$ is the least common ancestor of $z_k$ and $z_{k+1}$ as well as of $y_k$ and $z_{k+1}$. Thus $l(T[\beta_3,z_{k+1}]) \geq 2k$ and so the union of $T[\beta_3,z_{k+1}], T[\beta_3, y_k], Q_1[y_0,y_k]$ and $Q_2[y_0,z_{k+1}]$ forms a $S$-$C(k,1,k,1)$ in $D$, a contradiction. This finishes the proof of Assertion \ref{assertion2}. $\hfill {\lozenge}$ \medbreak
  \noindent In view of Assertion \ref{assertion2}, we get that  $x_1=x_2$. In what follows, we denote by $T_1=T[x_1,y_1] \cup T[x_1,z_1] \cup T[x_1,y_0]$. 
  \begin{assertion}\label{assertion3}
  For all $i \in [k]$ and $j \in [2k]$, $z_i$ is not an ancestor of $z_j$. 
  \end{assertion}
  \noindent \sl {Proof of Assertion \ref{assertion3}. }\upshape Assume the contrary is true, that is, there exists $i \in [k]$ and $j \in [2k]$ such that $z_i \leqslant_{T} z_j$. Assume that $i$ is chosen to be minimal and $j$ is chosen  so that $l_{T}(z_j)$ is maximal and $T[z_i,z_j] \cap Q_2=\{z_i,z_j\}$. Clearly, $z_{i+1} \notin T[z_i,z_j]$. Let $\alpha_1$ be the vertex of $T[r,z_{i+1}]$ such that $T[\alpha_1,z_{i+1}] \cap T_1 =\{\alpha_1\}$ if $x_1 \in T[r,z_{i+1}]$ and $\alpha_1=\textrm{l.c.a}\{x_1, z_{i+1}\}$ otherwise. Due to the choice of $z_i$ together with Assertion \ref{assertion1}, keep in mind that $T[\alpha_1, z_{i+1}] \cap (Q_1[y_1,y_{2k-1}] \cup Q_2[z_1,z_i]) =\emptyset$. 
  \begin{claim}\label{claim6}
  	$\alpha_1 \in T]x_1,y_1].$ 	
  \end{claim}
\noindent \sl {Subproof. }\upshape  Assume  the contrary is true. First, assume that $\alpha_1=y_0$.   Let $\alpha_2$ be the vertex of $T[r,y_2]$ such that $T[\alpha_2,y_2] \cap (T_1 \cup T[y_0,z_{i+1}] \cup T[z_i,z_j])=\{\alpha_2\}$ if $x_1 \in T[r,y_2]$ and $\alpha_2=\textrm{l.c.a}\{x_1,y_2\}$ otherwise. Note that $T[r,y_2] \cap T[z_i,z_j] =\emptyset$, since otherwise $z_i$ would be an ancestor of $y_2$, a contradiction to Assertion \ref{assertion1}.    If $\alpha_2 \notin T[y_0,z_{i+1}]$, we consider two cases: If $\alpha_2 \in T]x_3,z_1]$,  then the union of $T[\alpha_2,z_1] \cup Q_2[z_1,z_i] \cup T[z_i,z_j], T[\alpha_2,y_2], T[x_1,y_1] \cup (y_1,y_2)$ and $T[x_1,z_{i+1}] \cup Q_2[z_{i+1},z_j]$ forms a $S$-$C(k,1,k,1)$, a contradiction. Otherwise, consider $\beta=\textrm{min}_T\{\alpha_2,x_3\}$ if $\alpha_2$ and $x_3$ are ancestors and $\beta=\textrm{l.c.a}\{\alpha_2, x_3\}$ otherwise. Then the union of   $T[\beta,y_2], T[\beta,z_1]\cup Q_2[z_1,z_i] \cup T[z_i,z_j], T[y_0,z_{i+1}] \cup Q_2[z_{i+1},z_j]$ and $ Q_1[y_0,y_2]$ forms a $S$-$C(k,1,k,1)$, a contradiction. Else if $\alpha_2 \in T[y_0,z_{i+1}]$, then $\alpha_2 \neq y_0$, since otherwise the union of $T[y_0,y_2], T[y_0,z_{i+1}] \cup Q_2[z_{i+1},z_j], T[x_3,z_1] \cup Q_2[z_1,z_i] \cup T[z_i,z_j]$ and $ T[x_3,y_1] \cup (y_1,y_2)$  is a $S$-$C(k,1,k,1)$ in $D$, a contradiction. Moreover, $l(T[\alpha_2, y_2])\leq k-1$ and $l(T[\alpha_2,z_{i+1}]) \leq k-1$, since otherwise the union of $T[\alpha_2,y_2], T[\alpha_2,z_{i+1}] \cup Q_2[z_{i+1},z_j], T[x_1,y_1] \cup (y_1,y_2)$ and $T[x_1,y_0] \cup Q_2[y_0,z_i] \cup T[z_i,z_j]$ is a $S$-$C(k,1,k,1)$ in $D$, a contradiction. Furthermore, $Q_2[z_{i+1},z_{j}]$ has length at most $k-2$, since otherwise the union of $T[\alpha_2,z_{i+1}] \cup Q_2[z_{i+1},z_j], T[\alpha_2,y_2], T[x_1,y_1] \cup (y_1,y_2)$ and $T[x_1,y_0] \cup Q_2[y_0,z_i] \cup T[z_i,z_j]$ is a $S$-$C(k,1,k,1)$ in $D$, a contradiction.  This induces the existence of $z_{j+1}$. Let  $\alpha_3$ be the vertex of $T[r,z_{j+1}]$ such that $T[\alpha_3,z_{j+1}] \cap (T_1 \cup T[\alpha_2,y_2] \cup T[y_0,z_{i+1}]\cup T[z_i,z_j])=\{\alpha_3\}$ if $x_1 \in T[r,z_{j+1}]$ and $\alpha_3=\textrm{l.c.a}\{x_1, z_{j+1}\}$ otherwise. Observe that $T[\alpha_3, z_{j+1}] \cap Q_2]y_0,z_j] =\emptyset$. In fact, $T[\alpha_3, z_{j+1}] \cap Q_2]y_0,z_{i-1}] =\emptyset$, due to the choice of $z_i$. Moreover, $T[\alpha_3, z_{j+1}] \cap Q_2[y_{i+1},z_{j-1}] =\emptyset$, since otherwise if there exists $i+1 \leq t \leq j-1$ such that $z_{t}\leqslant_{T} z_{j+1}$, then the union of $T[\alpha_2,z_{i+1}] \cup Q_2[z_{i+1},z_t] \cup T[z_t, z_{j+1}], T[\alpha_2,y_2], T[x_1,y_1] \cup (y_1,y_2)$ and $T[x_1,y_0] \cup Q_2[y_0,z_i] \cup T[z_i,z_j] \cup (z_j,z_{j+1})$  is a $S$-$C(k,1,k,1)$, a contradiction. This  together with the maximality of $z_j$ imply that $z_i$ is also not an ancestor of $z_{j+1}$. To reach the final contradiction, we study the possible positions of $\alpha_3$:  If $\alpha_3 \in T]\alpha_2,y_2]$, then $l(T[\alpha_3,z_{j+1}]) \geq k$, since otherwise  $l(T[\alpha_3,y_2])\geq k$, a contradiction. Consequently, the union of $T[\alpha_3,z_{j+1}], T[\alpha_3,y_2], T[x_1,y_1] \cup (y_1,y_2)$ and $T[x_1,y_0]  \cup Q_2[y_0,z_i] \cup T[z_i,z_j] \cup (z_j,z_{j+1})$ is a $S$-$C(k,1,k,1)$, a contradiction. Else if $\alpha_3 \notin T]\alpha_2,y_2]$ and $l_T(\alpha_3)> l_T(y_0)$, then the union of $T[\beta,z_{j+1}], T[\beta,y_2], T[x_1,y_1] \cup (y_1,y_2)$ and $T[x_1,y_0]  \cup Q_2[y_0,z_i] \cup T[z_i,z_j] \cup (z_j,z_{j+1})$ is a $S$-$C(k,1,k,1)$ in $D$, with $\beta=\textrm{min}_T\{\alpha_2,\alpha_3\}$. This is a contradiction to the fact that $D$ is $C(k,1,k,1)$-subdivision-free. Else, let $\beta=\textrm{min}_T\{\alpha_3,y_0\}$ if $\alpha_3$ and $y_0$ are ancestors and let $\beta=\textrm{l.c.a}\{\alpha_3, y_0\}$ otherwise. Note that possibly  $\alpha_3=y_0$.  Hence, the union of $T[\beta,z_{i+1}], T[\beta,z_{j+1}], T[z_i,z_{j}] \cup (z_j,z_{j+1})$ and $(z_i,z_{i+1})$ is a $S$-$C(k,1,k,1)$ in $D$, a contradiction affirming that $y_0$ is not an ancestor of $z_{i+1} $ and thus $\alpha_1 \neq y_0$. But $\alpha_1 \notin T]x_1,y_1]$, then $x_1$ and $\alpha_1$ are ancestors. Consequently, the union of $T[\beta,y_1], T[\beta,z_{i+1}] \cup Q_2[z_{i+1},z_j],  Q_2[y_0,z_i] \cup T[z_i,z_{j}]$ and $(y_0,y_1)$ is a $S$-$C(k,1,k,1)$ in $D$, with $\beta=\textrm{min}_T\{x_1,\alpha_1\}$. This a contradiction to the fact that $D$ is $C(k,1,k,1)$-subdivision-free and thus a confirmation to our claim. $\hfill{\blacklozenge} $ \medbreak 
\noindent  Now we consider the vertex  $\alpha_2$ of  $T[r,y_2]$ such that $T[\alpha_2,y_2] \cap (T_1 \cup T[\alpha_1,z_{i+1}] \cup T[z_i,z_j])=\{\alpha_2\}$ if $x_1 \in T[r,y_{2}]$ and $\alpha_2=\textrm{l.c.a}\{x_1, y_2\}$ otherwise.  Note that $T[\alpha_2, y_2] \cap (Q_2]y_0,z_{j}] \cup T[z_i,z_j]) =\emptyset$, according to Assertion \ref{assertion1}. 
\begin{claim}\label{claim5}
$\alpha_2 \in T]\alpha_1,z_{i+1}]$. 	
\end{claim}
\noindent \sl {Subproof.  }\upshape   Assume the contrary is true.  First, assume that $\alpha_2=y_0$.   If $\alpha_1 \in T]x_1,x_3]$,  then the union of $T[\alpha_1,z_{i+1}], T[\alpha_1, y_1] \cup (y_1,y_2), T[y_0,y_2]$ and $Q_2[y_0,z_{i+1}]$ forms a $S$-$C(k,1,k,1)$, a contradiction.   Else if $\alpha_1 \in T]x_3,y_1]$, then  $l(T[\alpha_1,z_{i+1}]) \leq k-1$, since otherwise the union of $ T[\alpha_1,z_{i+1}],  T[\alpha_1,y_1] \cup (y_1,y_2),  T[y_0,y_2]$ and $Q_2[y_0,z_{i+1}]$  is a $S$-$C(k,1,k,1)$ in $D$, a contradiction. This implies that $l(T[x_3, \alpha_1]) \geq k$.  Moreover, $Q_2[z_{i+1},z_{j}]$ has length at most $k-2$, since otherwise the union of $T[\alpha_1,z_{i+1}] \cup Q_2[z_{i+1},z_j],$ $ T[\alpha_1,y_1] \cup (y_1,y_2), T[y_0,y_2]$ and $Q_2[y_0,z_i] \cup T[z_i,z_j]$ is a $S$-$C(k,1,k,1)$ in $D$, a contradiction.  This guarantees the existence of $z_{j+1}$. Let  $\alpha_3$ be the vertex of $T[r,z_{j+1}]$ such that $T[\alpha_3,z_{j+1}] \cap (T_1 \cup T[y_0,y_2] \cup T[\alpha_1,z_{i+1}] \cup T[z_i,z_j])=\{\alpha_3\}$ if $x_1 \in T[r,z_{j+1}]$ and $\alpha_3=\textrm{l.c.a}\{x_1, z_{j+1}\}$ otherwise. Observe that $T[\alpha_3, z_{j+1}] \cap Q_2]y_0,z_j] =\emptyset$. In fact, $T[\alpha_3, z_{j+1}] \cap Q_2]y_0,z_{i-1}] =\emptyset$, due to the choice of $z_i$. Moreover, $T[\alpha_3, z_{j+1}] \cap Q_2[z_{i+1},z_{j-1}] =\emptyset$, since otherwise if there exists $i+1 \leq t \leq j-1$ such that $z_{t}\leqslant_{T} z_{j+1}$, then the union of $T[\alpha_1,z_{i+1}] \cup Q_2[z_{i+1},z_t] \cup T[z_t, z_{j+1}], T[\alpha_1,y_1] \cup (y_1,y_2), T[y_0,y_2]$ and $Q_2[y_0,z_i] \cup T[z_i,z_j]  \cup (z_j,z_{j+1})$  is a $S$-$C(k,1,k,1)$, a contradiction. This  together with the maximality of $z_j$ imply that $z_i$ is also not an ancestor of $z_{j+1}$. To reach the final contradiction, we study the possible positions of $\alpha_3$:  If $\alpha_3 \in T[y_0,y_2]$, then the union of $T[y_0, z_{j+1}], (y_0,y_1), T[x_3,y_1 ], T[x_3,z_1] \cup Q_2[z_1, z_{j+1}]$ is a $S$-$C(k,1,k,1)$ in $D$, a contradiction. Else if $\alpha_3 \notin T]\alpha_1,y_1]$, let $\beta=\textrm{min}_T\{\alpha_1, \alpha_3\}$ if $\alpha_1$ and $\alpha_3$ are ancestors and let $\beta=\textrm{l.c.a}\{\alpha_1, \alpha_3\}$ otherwise.  Then the union of $T[\beta, z_{j+1}], T[\beta, y_1] \cup (y_1,y_2), T[y_0,y_2]$ and $Q_2[y_0,z_{j+1}]$ is a $S$-$C(k,1,k,1)$, a contradiction. Else, $\alpha_3 \in T]\alpha_1,y_1]$. Note that $l(T[\alpha_3, z_{j+1}]) \leq k-1$, since otherwise the union of $T[\alpha_3, z_{j+1}], T[\alpha_3, y_1] \cup (y_1,y_2), T[y_0,y_2]$ and $Q_2[y_0,z_{j+1}]$ is a $S$-$C(k,1,k,1)$, a contradiction. But $l(T[\alpha_1, z_{j+1}]) \geq 2k$, then $l(T[\alpha_1, \alpha_3]) \geq k$ and so the union of $T[\alpha_1,y_1], T[\alpha_1,z_{i+1}] \cup Q_2[z_{i+1},z_j], (y_0,y_1)$ and $Q_2[y_0,z_i] \cup T[z_i,z_j]$ is a $S$-$C(k,1,k,1)$, a contradiction. This confirms that $\alpha_2 \neq y_0$. To complete the proof, we assume to the contrary that $\alpha_2 \notin T]\alpha_1,z_{i+1}]$ and we consider $\beta=\textrm{min}_T\{\alpha_1,\alpha_2\}$ if $\alpha_1$ and $\alpha_2$ are ancestors and  $\beta=\textrm{l.c.a}\{\alpha_1, \alpha_2\}$ otherwise. Then the union of $T[\beta, y_2], T[\beta, z_{i+1}] \cup Q_2[z_{i+1},z_j], Q_1[y_0,y_2]$ and $Q_2[y_0,z_i] \cup T[z_i,z_j]$ is a $S$-$C(k,1,k,1)$, a contradiction.  This implies Claim \ref{claim5}.  $\hfill{\blacklozenge} $ \medbreak


 \noindent The rest of the proof of Assertion \ref{assertion3} is exactly the same as $*\{...\}*$ in Claim \ref{cl2}, with exactly two differences. The first difference is that each place we have used $T[x_2,z_1] \cup Q_2[z_1,z_i]$ in $*\{...\}*$    must be replaced by $T[x_2,y_0] \cup Q_2[y_0, z_i ]$ in the proof of Assertion \ref{assertion3}.  The second one is that  in the proof of Claim \ref{cl2}  $T[\alpha_3, z_{j+1}] \cap Q_2[y_0,z_{i-1}] =\emptyset$ due to the induction hypothesis. However,  in the proof of this assertion we have $T[\alpha_3, z_{j+1}] \cap Q_2]y_0,z_{i-1}] =\emptyset$ by the choice of $z_i$. Indeed, the case where $y_0 \in T[\alpha_3,z_{j+1}]$ in the proof of this assertion will be prevented by the last contradiction of Claim \ref{cl2}. This ends the proof.  $\hfill {\lozenge}$ \medbreak 
  \noindent In a similar way, we can show that  $y_i$ is not an ancestor of $y_j$ for all $i \in [k]$ and $j \in [2k]$. To complete the proof, let $\alpha_1$ (resp. $\alpha_2$)  be the vertex  of  $T[r,y_k]$ (resp. $T[r,z_k]$) such that $T[\alpha_1,y_k] \cap T_1 =\{\alpha_1\}$ (resp. $T[\alpha_2,z_k] \cap T_1 =\{\alpha_2\}$) if $x_1 \in T[r,y_k]$ (resp. $x_1 \in T[r,z_k]$) and $\alpha_1=\textrm{l.c.a}\{x_1,y_k\}$ (resp. $\alpha_2=\textrm{l.c.a}\{x_1,z_k\}$) otherwise. 
  \begin{assertion}\label{assertion4}
  	$y_0 \in \{\alpha_1,\alpha_2\}.$
  \end{assertion}
\noindent \sl {Proof of Assertion \ref{assertion4}. }\upshape Assume the contrary is true. This implies together with Assertion \ref{assertion1} and Assertion \ref{assertion3} that $(T[\alpha_1,y_k] \cup T[\alpha_2,z_k]) \cap (Q_1[y_0,y_k] \cup Q_2[y_0,z_k])=\emptyset.$ Set $\alpha_3=\textrm{l.c.a}\{y_k,z_k\}$ and $\alpha_4 =\textrm{min}_T\{\alpha_1,\alpha_2\}$ if $\alpha_1$ and $\alpha_2$ are ancestors and $\alpha_4=\textrm{l.c.a}\{\alpha_1,\alpha_2\}$ otherwise. Assume that $\alpha_3=\alpha_4$, then $l(T[\alpha_1,y_k]) \geq k$ and $l(T[\alpha_2,z_k]) \geq k$ unless $\alpha_1 \in T]x_3,z_1]$ and $\alpha_2 \in T]x_3,y_1]$.  In the later case, $\alpha_3=x_3$ and so $l(T[x_3,z_k]) \geq k$ and $l(T[x_3,y_k]) \geq k.$ Then the union of $T[\alpha_3, z_k], T[\alpha_3,y_k], Q_1[y_0,y_k]$ and $Q_2[y_0,z_k]$ is a $S$-$C(k,1,k,1)$, a contradiction. This implies that $\alpha_3 \neq \alpha_4$ and thus $\alpha_1=\alpha_2=\alpha_4$ and $\alpha_3 \in T[\alpha_1, z_k]$. Clearly, $T[\alpha_3,y_k]$ and $T[\alpha_3,z_k]$ have length at most $k-1$, since otherwise the union of $T[\alpha_3, y_k], T[\alpha_3,z_k], Q_2[y_0,z_k]$ and $Q_1[y_0,y_k]$ is   a $S$-$C(k,1,k,1)$, a contradiction. This gives that $l(T[\alpha_1, \alpha_3 ]) \geq k$, since at least one of $T[\alpha_1,y_k]$ and $T[\alpha_1,z_k]$ has length at least $2k$.  Let  $\alpha_5$  be the vertex  of  $T[r,z_{k+1}]$  such that $T[\alpha_1,z_{k+1}] \cap ( T_1 \cup T[\alpha_1,y_k] \cup T[\alpha_1,z_k]) =\{\alpha_5\}$  if $x_1 \in T[r,z_{k+1}]$ and $\alpha_5=\textrm{l.c.a}\{x_1,z_{k+1}\}$  otherwise. Due to Assertion \ref{assertion1} and Assertion \ref{assertion3},  it follows that $T[\alpha_5,z_{k+1}] \cap (Q_1[y_1,y_k] \cup Q_2[z_1,z_k]) =\emptyset$. If $\alpha_5=y_0$, let  $\beta =\alpha_1$ if $\alpha_1 \leqslant_{T} \beta_1$  and $\beta=\textrm{l.c.a}\{\alpha_1,z_1\}$ otherwise. Then the union of $Q_1[y_0,y_k], T[y_0,z_{k+1}], T[\beta, z_1] \cup Q_2[z_1,z_{k+1}]$ and $T[\beta,y_k]$ is a $S$-$C(k,1,k,1)$, a contradiction.  Hence, $\alpha_5\neq y_0$. Let  $\beta =\alpha_5$ if $\alpha_5 \leqslant_{T} y_k$  and $\beta=\textrm{l.c.a}\{\alpha_5,y_k\}$ otherwise. Then $l(T[\beta,z_{k+1}]) \geq k$ if $\alpha_5 \in T_{\alpha_1}$ and $l(T[\beta,y_k]) \geq k$ otherwise. This implies that the union of  $Q_1[y_0,y_k], Q_2[y_0,z_{k+1}], T[\beta, z_{k+1}]$ and $T[\beta,y_k]$ is a $S$-$C(k,1,k,1)$, a contradiction. This ends the proof. $\hfill {\lozenge}$ \medbreak 
\noindent To reach the final contradiction, we consider two principle cases: If $\alpha_1=\alpha_2=y_0$, then the union of $T[y_0,z_k], T[y_0,y_k], T[x_3,y_1] \cup Q_1[y_1,y_k]$ and $T[x_3,z_1] \cup Q_2[z_1,z_k]$is a $S$-$C(k,1,k,1)$, a contradiction. Otherwise,  due to symmetry, assume that $\alpha_1=y_0$ and $\alpha_2 \neq y_0$. If $\alpha_2 \notin T[x_3,z_1] \cup T[x_3,y_1]$, let $\beta=\alpha_2$ if $\alpha_2 \leqslant_{T} y_0$ and $\beta= \textrm{l.c.a}\{\alpha_2, y_0\}$ otherwise. Then the union of $T[\beta,y_k], T[\beta, z_k], T[x_3,z_1] \cup Q_2[z_1,z_k]$ and $T[x_3, y_1] \cup Q_1[y_1,y_k]$ is a $S$-$C(k,1,k,1)$, a contradiction. Hence, $\alpha_2 \in T[x_3,z_1] \cup T[x_3,y_1]$. Let $\beta=\alpha_2$ if $\alpha_2 \leqslant_{T} y_1$ and $\beta= x_3$ otherwise. Then the union of $T[\beta,y_1] \cup Q_1[y_1,y_k], T[\beta, z_k],  Q_2[y_0,z_k]$ and $T[y_0,y_k]$ is a $S$-$C(k,1,k,1)$, a contradiction. This finishes the proof. $\hfill {\square}$
  \end{proof}
\subsection{Main Theorem}
\vspace{1.5mm} Now we are ready to state our main theorem:
 \begin{theorem}
 		Let $D$ be a strongly connected digraph having no subdivisions of  $C(k_1,1,k_3,1)$ and let  $k= \textrm{max} \{k_1, k_3\}$, then the chromatic number of $D$ is at most $36. (2k) .(4k+2)$.
 \end{theorem}
\begin{proof}
Let $T$ be a spanning out-tree of $D$. Indeed, the existence of $T$ is guaranteed due to the fact that  $D$ is strongly connected digraph. According to Proposition \ref{finaltree}, we may assume that $T$ is final. Define $D^i_j$ as before for $i \in [2k]$ and $j \in [3]$. Due to Lemma \ref{far} together with Proposition \ref{prop1}, Proposition \ref{p2} and Proposition \ref{p3},  we get that $\chi(D_i)\leq 36.(4k+2)$ for all $i \in [2k]$.  As $V(D)=\bigcup _{i=1}^{2k} V(D_i)$,  by assigning $36(4k-2)$ distinct colors to each $D_i$, we obtain a proper coloring of $D$ with $36. (2k) .(4k+2)$ colors. $\hfill {\square}$
\end{proof}

\section{The existence of $S$-$C(k_1,1,k_3,1)$ in Hamiltonian digraphs}

\noindent The previous  bound can be strongly improved  in case that the digraph contains a Hamiltonian directed cycle. In this section, we  provide a tighter bound for the chromatic number of Hamiltonian digraphs containing no subdivisions of $C(k,1,k,1)$.  Before stating the main theorem of this section, we need the following lemma: 
\begin{lemma}\label{property1}
	Let $k$ be a positive integer and let $D$ be a $C(k,1,k,1)$-subdivision-free digraph with a Hamiltonian directed cycle $C$. Assume that $u,v,w,x,x'$ are five vertices of $D$ such  that $uv \in E(G(D))\setminus E(C), w \in C]u,v[$ and $x,x' \in C]v,u[$  in a way that  $v(C[v,x])=k$ and $v(C[x',u])=k$.  If $(v,u) \in E(D)$, then  $|N_{G(D)}(w) \cap C]x,x'[| \leq 2.$
\end{lemma}
\begin{proof}
	We are going to prove first that $w$ has at most one out-neighbor in  $V(C]x,x'[)$.  If $w$ has two out-neighbors in  $V(C]x,x'[)$, say $v_i, v_j$ with $i<j$, then the union of  $(w,v_j) \cup C[v_j,u]$, $(w,v_i)$, $C[v, v_i]$ and $(v,u)$ forms a $S$-$C(k,1,k,1)$ in $D$, a contradiction. Thus, \begin{align}\label{1}
	|N^+_{D}(w) \cap C]x,x'[| \leq 1.
	\end{align}
	 Now we will prove that  $w$ has at most one in-neighbor in  $V(C]x,x'[) $.    If  $w$ has two in-neighbors in  $V(C]x,x'[)$, say $v_i, v_j$ with $i<j$, then the union of $C[v, v_i] \cup (v_i,w)$, $(v,u)$, $C[v_j,u]$ and $(v_j,w)$ is a $S$-$C(k,1,k,1)$, a contradiction. Thus, \begin{align}\label{2} |N^-_{D}(w) \cap C]x,x'[| \leq 1.
	 \end{align}
	 Hence, according to the inequalities \ref{1} and \ref{2}, we get that $|N_{G(D)}(w) \cap C]x,x'[| \leq 2$. This ends the proof.  $\hfill {\square}$
\end{proof}
\begin{theorem}\label{hamilt}
	Let $D$ be a Hamiltonian digraph having no subdivisions of  $C(k_1,1,k_3,1)$ and let  $k= \textrm{max} \{k_1, k_3\}$. Then $D$ is $(6k-1)$-degenerate and thus $\chi(D) \leq 6k$.
\end{theorem}
\begin{proof}
	Let $G$ be a subgraph of $G(D)$ and let $H$ be the subdigraph of $D$ whose underlying graph is $G$.  If $\delta(G) \leq 6k-1$, then we are done. Otherwise, we will prove that $D$  contains a $S$-$C(k,1,k,1)$, which means that the case where $\delta(G) > 6k-1$ does not hold. Thus it suffices now to prove that if $\delta(G) \geq 6k$ for a subgraph $G$ of $G(D)$, then $D$ contains a $S$-$C(k,1,k,1)$. Suppose the contrary is true and let $C=v_1 v_2 ...v_n $ be a Hamiltonian directed cycle of $D$, where $n=|V(D)| \geq |V(G)| \geq \delta(G)+1 \geq 6k+1$. 
	 Since  $\delta(G) \geq 6k$, then there exist two vertices  $u,v$  of $G$ such that $uv \in E(G) \setminus E(C)$ and $|V(C[u,v]) \cap V(G) | \geq 3k+1$. Assume that $u,v$ are chosen such that $|V(C[u,v]) \cap V(G) |$ is minimal but at least $3k+1$. This implies that $|N_G(u) \cap V(C[u,v])| =3k$. Hence, $|N_G(u) \cap V(C]v,u[)| \geq 3k$, since otherwise we get that $d_G(u) \leq 6k-1 < \delta(G)$, a contradiction.  Thus, we guarantee the existence of two distinct vertices $t$ and $t'$ of $C]v,u[$ such that $l(C[v,t])=k-1$ and $l(C[t',u])=k-1$.\\
	 \noindent  Now we will consider the possible directions of the edge $uv$ in $H$. If $(v,u) \in E(H)$, we define $v'$ to  be the vertex of $G$ such that $C[v',v] \cap V(G) =\{v',v\}$.  By the choice of the edge $uv$,  note that $v'$ has at most $3k-1$ neighbors in $C[u, v'] \cap V(G)$ and thus  $ |N_G(v') \cap C[u,v]| \leq 3k$. Moreover, Lemma \ref{property1} gives that   $|N_G(v') \cap C]t,t'[| \leq 2$.  Combining all these together, we get 
 \begin{align*} d_G(v')  & =|N_G(v') \cap C[u,v]| + |N_G(v') \cap C]v,t]|+  |N_G(v') \cap C]t,t'[|+ |N_G(v) \cap C[t',u[|  \\
  &\leq 3k+(k-1)+2+(k-1) \\
& =5k, \end{align*}
 contradicting the fact that $\delta(G) \geq 6k$. Therefore, $(v,u) \notin E(H)$ and so $(u,v) \in E(H)$. \medbreak
 
 \noindent Now we consider the vertices 	$w$ and $w'$ of $G$ with $|V(C[u,w]) \cap V(G)|=k+1$ and $|V(C[w',v]) \cap V(G)|=k+1$.  Due to the fact that $|V(C[u,v]) \cap V(G) | \geq 3k+1$, it is clear that  $w \neq w'$ and   $|V(C[w,w']) \cap V(G)| \geq k+1$. To reach the final contradiction, we need to prove a series of claims as follows.
	\begin{claim}\label{claim1}
If $N^+_G(w') \cap C]t,t'[ \neq \emptyset$, then  $N_G(w) \cap C]t,t'[ = \emptyset$.
	\end{claim}
\noindent \sl {Subproof.} \upshape Let $v_i$ be the out-neighbor of $w'$ in $G \cap C]t,t'[$ such that $i$ is minimal. We are going to show now that $|N^+_G(w) \cap C]t,t'[| =0$. Suppose not and let $v_j \in |N^+_G(w) \cap C]t,t'[$. If $i \leq j$, then the union of $C[w',v]$, $(w',v_i) \cup C[v_i,v_j]$, $C[u,w] \cup(w,v_j)$ and $(u,v)$ is a $S$-$C(k,1,k,1)$ in $D$, a contradiction.  Otherwise, the union of $C[w',v]$, $(w',v_i) $, $C[u,w] \cup(w,v_j) \cup  C[v_j,v_i]$ and $(u,v)$ is a $S$-$C(k,1,k,1)$, a contradiction. This proves that $N^+_G(w) \cap C]t,t'[=\phi $. Now we shall show that $|N^-_G(w) \cap C]t,t'[| =0$. Suppose not and let $v_j \in |N^-_G(w) \cap C]t,t'[$. If $ i \leq j$, then the union of $C[w',v]$, $(w',v_i) \cup C[v_i,v_j] \cup (v_j,w)$, $C[u,w]$ and $(u,v)$ is a $S$-$C(k,1,k,1)$, a contradiction. Otherwise, we are going to argue on the neighbors of $w'$ in $G$. First,  consider the directed cycle $C[w,v_j] \cup (v_j,w)$. Since $|V(C[v_j,w])|\geq |V(C[v_j,v_i])|+|V(C[t',u[)|+|V(C[u,w])| \geq 1+(k-1)+(k+1)=2k+1$, then $w'$ has at most 2 neighbors in $C]v_{j+k-1}, u[ \cap G$, due to Lemma \ref{property1}. Consequently, it follows that $w'$ has at most $k+1$ neighbors in $C]v_{j}, u[ \cap G$. Moreover, $w'$ has no neighbors in $C]t,v_j] \cap G$. In fact,  by the choice of the out-neighbor $v_i$ of $w'$, it is clear to see that $w'$ has no out-neighbors in $C]t,v_j] \cap G$.  Also,  $w'$ has no in-neighbors in $C]t,v_j] \cap G$, since otherwise the union of $(z,w') \cup C[w',v]$, $C[z,v_j] \cup (v_j,w) $, $C[u,w]$ and $(u,v)$ is a $S$-$C(k,1,k,1)$ in $D$, with $z$ is an in-neighbor of $w'$ in  $C]t,v_j] \cap G$. This is a contradiction to the fact that $D$ is $C(k,1,k,1)$-subdivision-free. Furthermore, by the choice of the edge $uv$, note that $w'$ has at most $3k-1$ neighbors in $C[u,w'] \cap G$. This  together with the fact that $|V(C[w',v]) \cap V(G)|=k+1$  imply that $w'$ has at most $4k-1$ neighbors in $C[u,v] \cap G$. Therefore, according of all what precedes, we get 
 \begin{align*} d_G(w')  & =|N_G(w') \cap C[u,v]| +   |N_G(w') \cap C]v,t]|+ |N_G(w') \cap C]t,v_j]|+ |N_G(w') \cap C]v_j+u[| \\
 &\leq (4k-1)+(k-1)+0+(k+1) \\
 & =6k-1, \end{align*}
a contradiction to the fact that $\delta(G) \geq 6k$, affirming our claim.  $\hfill {\blacklozenge}$
\begin{claim}\label{claim2}
	$|N^+_G(w') \cap C]t,t'[|=0.$
\end{claim}
\noindent \sl {Subproof.} \upshape Suppose to the contrary that  $w'$ has an out-neighbor in $G \cap C]t,t'[$. Thus, according to Claim \ref{claim1}, we get that $|N^+_G(w) \cap C]t,t'[|=0$. Hence, \begin{align*} d_G(w)  & =|N_G(w) \cap C[u,w[| + |N_G(w) \cap C]w,v]|+ |N_G(w) \cap C]v,t]|+ |N_G(w) \cap C]t,t'[|+ |N_G(v) \cap C[t',u[|  \\
&\leq k+(3k-1)+(k-1)+0+(k-1) \\
& =6k-3, \end{align*}
a contradiction to the fact that   $\delta(G) \geq 6k$.  This proves our claim. $\hfill {\blacklozenge}$
\begin{claim}\label{claim3}
If $N^-_H(w') \cap C]t,t'[ \neq \emptyset$, then  $N^-_H(w) \cap C]t,t'[ = \emptyset$.
\end{claim}
\noindent \sl {Subproof.} \upshape Suppose the contrary is true and let  $v_i \in |N^-_H(w) \cap C]t,t'[$ such that $i$ is minimal.   Let $v_j$ be an in-neighbor of $w'$ in $H \cap C]t,t'[$. If $i \leq j$,   then the union of $C[v_i,v_j] \cup (v_j,w') \cup C[w',v]$, $(v_i,w)$, $C[u,w]$ and $(u,v)$ is a $S$-$C(k,1,k,1)$ in $D$, a contradiction. Otherwise, the union of  $(v_j,w') \cup C[w',v]$, $C[v_j,v_i] \cup (v_i,w)$, $C[u,w]$ and $(u,v)$ is a $S$-$C(k,1,k,1)$ in $D$, a contradiction. This confirms Claim \ref{claim3}. $\hfill {\blacklozenge}$
\begin{claim}\label{claim4}
	If $|N^-_H(w') \cap C]t,t'[ | \geq 2$, then  $N^+_H(w) \cap C]t,t'[ = \emptyset$.
\end{claim}
\noindent \sl {Subproof.} \upshape 	Since $|N^-_H(w') \cap C]t,t'[ | \geq 2$, then there exist two distinct vertices   $z$ and $z'$ in $N^-_H(w') \cap C]t,t'[$. Assume that $z,z'$ are chosen so that $z=v_i$ such that $i$ is minimal and $z'=v_j$ for some $j>i$. Suppose now that  $N^+_H(w) \cap C]t,t'[ \neq \emptyset$ and let $v_p$ be an out-neighbor of $w$ in $C]t,t'[ \cap H$. If $p \geq j$, then the union of $(v_i,w') \cup C[w',v]$,  $C[v_i,v_p]$, $C[u,w] \cup (w,v_p)$ and  $(u,v)$ is a $S$-$C(k,1,k,1)$ in $D$, a contradiction. Otherwise, the union of $(v_j,w') \cup C[w',v]$, $C[v_j,u] \cup (u,v) \cup C[v,v_p]$, $C[w,w']$ and $(w,v_p)$ forms a $S$-$C(k,1,k,1)$, a contradiction.  This proves Claim \ref{claim4}. $\hfill {\blacklozenge}$  \medbreak

\noindent To complete the proof, we are going to prove that $w'$ has at most one in-neighbor in $C]t,t'[ \cap H$. Suppose not, then Claim \ref{claim3}
and Claim \ref{claim4} imply that $w$ has no neighbors in $C]t,t'[ \cap G$. Hence, 
\begin{align*} d_G(w)  & =|N_G(w) \cap C[u,w[| + |N_G(w) \cap C]w,v]|+ |N_G(w) \cap C]v,t]|+ |N_G(w) \cap C]t,t'[|+ |N_G(v) \cap C[t',u[|  \\
&\leq k+(3k-1)+(k-1)+0+(k-1) \\
& =6k-3, \end{align*}
a contradiction to the fact that   $\delta(G) \geq 6k$. Thus, $|N^-_H(w') \cap C]t,t'[ | \leq 1$. Consequently, according to Claim \ref{claim3}, $|N_G(w') \cap C]t,t'[ | \leq 1$. Therefore, 
 \begin{align*} d_G(w')  & =|N_G(w') \cap C[u,v]| +   |N_G(w') \cap C]v,t]|+ |N_G(w') \cap C]t,t'[|+ |N_G(w') \cap C[t',u[| \\
&\leq (4k-1)+(k-1)+1+(k-1) \\
& =6k-2, \end{align*}
a contradiction to the fact that $\delta(G) \geq 6k$. This completes the proof. $\hfill {\square}$
\end{proof}

\end{document}